\documentclass[a4paper,11pt]{article}

\usepackage{xcolor}
\usepackage{lmodern}
\usepackage[T1]{fontenc}
\usepackage{hyperref}
\usepackage{amsfonts,amsmath,amssymb,amsopn,amsthm}
\usepackage[all]{xy}
\usepackage{vmargin}
\usepackage{makeidx}
\usepackage{graphicx}
\usepackage{tikz}
\usepackage{comment}
\usepackage{mathabx}
\title{Twisted differential operators and $q$-crystals}
\author{Michel Gros, Bernard Le Stum \& Adolfo Quir\'os\thanks{Supported by grant PGC2018-095392-B-I00 (MCIU/AEI/FEDER, UE).}
}
\date{Version of \today}

 \newcommand{\Addresses}{{
 \bigskip
 \footnotesize

Michel Gros, \textsc{IRMAR, Université de Rennes,
Campus de Beaulieu, 35042 Rennes cedex, France}\par\nopagebreak
\texttt{michel.gros@univ-rennes1.fr}

\medskip

Bernard Le Stum, \textsc{IRMAR, Université de Rennes,
Campus de Beaulieu, 35042 Rennes cedex, France}\par\nopagebreak
\texttt{bernard.le-stum@univ-rennes1.fr}

\medskip

Adolfo Quir\'os, \textsc{Departamento de Mat\'ematicas, Facultad de Ciencias, Universidad Aut\'onoma de Madrid, E-28049 Madrid, Spain}\par\nopagebreak
\texttt{adolfo.quiros@uam.es}

}}

\usepackage[ style=alphabetic,citestyle=alphabetic,sorting=nyt,sortcites=true,autopunct=true,babel=hyphen,hyperref=true,abbreviate=false,backref=true]{biblatex}
\AtEveryBibitem{ \clearfield{url} \clearfield{urldate} \clearfield{review} \clearfield{series} \clearfield{doi} \clearfield{isbn} \clearfield{issn} } 
\defbibheading{bibempty}{}
\newtheorem{thm}{Theorem}[section]
\newtheorem{prop}[thm] {Proposition}
\newtheorem{cor}[thm] {Corollary}
\newtheorem{lem}[thm] {Lemma}
\theoremstyle{definition}
\newtheorem{dfn}[thm] {Definition}

\newenvironment{xmp}[1][Example.]{\begin{trivlist} \item[\hskip \labelsep {\bfseries #1}]}{\end{trivlist}}
\newenvironment{xmps}[1][Examples.]{\begin{trivlist} \item[\hskip \labelsep {\bfseries #1}]}{\end{trivlist}}

\newenvironment{rmk}[1][Remark.]{\begin{trivlist} \item[\hskip \labelsep {\bfseries #1}]}{\end{trivlist}}
\newenvironment{rmks}[1][Remarks.]{\begin{trivlist} \item[\hskip \labelsep {\bfseries #1}]}{\end{trivlist}}
\begin{document}

\maketitle

\begin{abstract}
We discuss the notion of a $q$-PD-envelope considered by Bhatt and Scholze in their recent theory of $q$-crystalline cohomology and explain the relation with our notion of a divided polynomial twisted algebra.
Together with an interpretation of crystals on the $q$-crystalline site, that we call $q$-crystals, as modules endowed with some kind of stratification, it allows us to associate a module on the ring of twisted differential operators to any $q$-crystal.
For simplicity, we explain here only the one dimensional case.
\end{abstract}

\tableofcontents

\section*{Introduction}
\addcontentsline{toc}{section}{Introduction}

In their recent article \cite{BhattScholze22}, Bhargav Bhatt and Peter Scholze have introduced two new cohomological theories with a strong crystalline flavor, the prismatic and the $q$-crystalline cohomologies, allowing them to generalize some of their former results obtained with Matthew Morrow on $p$-adic integral cohomology. 
 As explained in (\cite{Gros20}, section 6), once a theory of coefficients for these new cohomologies is developed, these tools could also be a way for us to get rid of some non-canonical choices in our construction of the twisted Simpson correspondence (\cite{GrosLeStumQuiros19}, corollary 8.9).
At the end, this correspondence should hopefully appear just as an \emph{avatar} of a deeper canonical equivalence of crystals on the prismatic and the $q$-crystalline sites, whose general pattern should look like a ``$q$-deformation'' of Hidetoshi Oyama's reinterpretation (\cite{Oyama17}, theorem 1.4.3) of Ogus-Vologodsky's correspondence as an equivalence between categories of crystals (see \cite{Gros20}, section 6, for a general overview).
 
In this note, we start elaborating on this project by showing (theorem \ref{mainres}) how to construct a functor from the category of crystals on Bhatt-Scholze's $q$-crystalline site (that we call for short \emph{$q$-crystals}) to the category of modules over the ring $D_q$ of twisted differential operators of \cite{GrosLeStumQuiros19}, section 5.
The construction of this functor has also the independent interest of describing explicitly, at least locally, the kind of structure hiding behind a $q$-crystal.
Indeed, the ring $D_q = D_{A/R,q}$ is defined only for a very special class of algebras $A$ over a base ring $R$ and some element $q \in R$ (let's call them in this introduction, forgetting additional data, simply {\emph{twisted $R$-algebras}; see \cite{LeStumQuiros18b} or \cite{GrosLeStumQuiros19} for a precise definition) and it is obviously only for them that the functor will be constructed.
The construction consists then of two main steps.
The first one, certainly the less standard, is to describe concretely the \emph{$q$-PD-envelope} introduced a bit abstractly in \cite{BhattScholze22}, lemma 16.10, of the algebra of polynomials in one indeterminate $\xi$ over these twisted $R$-algebras and to relate it to the algebra appearing in the construction of $D_q$.
The second step is to develop the $q$-analogs of the usual calculus underlying the theory of classical crystals: hyperstratification, connection\ldots for the $q$-crystalline site associated to a twisted $R$-algebra.

Let us now describe briefly the organization of this article.
In section 1, we recall for the reader's convenience some basic vocabulary and properties of $\delta$-rings used in \cite{BhattScholze22}.
In section 2, we essentially rephrase some of our considerations about twisted powers \cite{GrosLeStumQuiros19}, using this time $\delta$-structures instead of relative Frobenius.
Following \cite{BhattScholze22}, we introduce in section 3 the notion of $q$-PD-envelope, first ignoring for it some hypotheses of completeness in loc.\ cit, and prove (theorem \ref{algun}) that our twisted divided power algebra in the indeterminate $\xi$ coincides exactly with this a priori completely different object.
Section 4 then addresses the general question of completion of $q$-PD-envelopes.
In section 5 we start the elaboration of twisted calculus and develop the language of hyper-$q$-stratifications (definition \ref{hyperqpd}), that is to say, the twisted variant of hyperstratifications, where new phenomena like non-trivial flip maps (proposition \ref{flippf}) appear.
Once we have obtained all these materials, we explain in section 6 how they can be reinterpreted in the language of $D_{q}$-modules and we establish (propositions \ref{taylD} and \ref{bigequi}) the twisted version of the classical equivalences of categories of usual calculus.
We have then all the ingredients to come back, in section 7, to the question of describing the complete $q$-PD-envelope of the diagonal embedding (theorem \ref{univpro}) and to give the construction of the functor (theorem \ref{mainres}) going from the category of $q$-crystals to that of $D_{q}$-modules in our particular geometric setting.

Everything below depends on a prime $p$, fixed once for all.
We stick to dimension one, leaving for another occasion the consideration of the higher dimensional case, which requires additional inputs.
Also, since we are mainly interested in local questions, we concentrate on the affine case.

We are extremely grateful and indebted to the referee for her/his very meticulous and constructive reading of the successive versions, pointing out in many places the need for further explanations or the existence of several issues, as well as for providing us with precious hints on how to fix them and suggestions to improve the whole presentation.

After the release of the first version of this article, two preprints on related matters have appeared and could be of great interest for the readers: one by Andre Chatzistamatiou \cite{Chatzistamatiou20}, the other by Matthew Morrow and Takeshi Tsuji \cite{MorrowTsuji20}.

The first author (M.G.) heartily thanks the organizers, Bhargav Bhatt and Martin Olsson, for their invitation to the {\emph{Simons Symposium on $p$-adic Hodge Theory}} (April 28-May 4, 2019), allowing him to follow
the progress on topics related to those treated in this note.

\section{$\delta$-structures}

In this section, we briefly review the notion of a $\delta$-ring.
We start from the Witt vectors point of view since it automatically provides all the standard properties.

As a set, the \emph{ring of ($p$-typical) Witt vectors} of length two\footnote{We use the more recent index convention for truncated Witt vectors.} on a commutative ring $A$ is $W_{1}(A) = A \times A$.
It is endowed with the unique \emph{natural} ring structure such that both the \emph{projection map}
\[
W_{1}(A) \to A, \quad (f, g) \mapsto f
\]
and the \emph{ghost map}
\[
W_{1}(A) \to A, \quad (f, g) \mapsto f^p+pg
\]
are ring homomorphisms.
A \emph{$\delta$-structure} on $A$ is a section
\[
A \to W_{1}(A), \quad f \mapsto (f, \delta(f))
\]
of the projection map in the category of rings.
This is equivalent to giving a $p$-derivation: a map $\delta : A \to A$ such that $\delta(0) = \delta(1) = 0$,
\begin{equation} \label{deltad}
\forall f,g \in A, \quad \delta(f+g) = \delta(f) + \delta(g) - \sum_{k=1}^{p-1} \frac 1p {p \choose k} f^{p-k}g^{k}
\end{equation}
and
\begin{equation} \label{deltmu}
\forall f,g \in A, \quad \delta(fg) =f^p\delta(g) + \delta(f)g^p +p\delta(f)\delta(g).
\end{equation}
Note that a $\delta$-structure is uniquely determined by the action of $\delta$ on the generators of the ring.
A \emph{$\delta$-ring} is a commutative ring endowed with a $\delta$-structure and $\delta$-rings make a category in the obvious way.
We refer the reader to Andr\'e Joyal's note \cite{Joyal85} for a short but beautiful introduction to the theory.
When $R$ is a fixed $\delta$-ring, a $\delta$-$R$-algebra is an $R$-algebra endowed with a compatible $\delta$-structure.

\begin{xmps}
\begin{enumerate}
\item There exists only one $\delta$-structure on the ring $\mathbb Z$ of rational integers which is given by
\[
\delta(n) = \frac {n-n^p}p \in \mathbb Z.
\]
\item If $f \in \mathbb Z[x]$, then there exists a unique $\delta$-structure on $\mathbb Z[x]$ such that $\delta(x) = f$.
\item There exists no $\delta$-structure at all when $p^kA = 0$, unless $A=0$ (show that $v_{p}(\delta(n)) = v_{p}(n) -1$ when $v_{p}(n)> 0$ and deduce that $v_{p}(\delta^k(p^k)) = 0$).
\item It may also happen that there exists no $\delta$-structure at all even when $A$ is $p$-torsion-free: take $p = 2$ and $A = \mathbb Z[\sqrt{-1}]$.
\end{enumerate}
\end{xmps}

A \emph{Frobenius} on a commutative ring $A$ is a (ring) morphism $\phi : A \to A$ that satisfies
\[
\forall f \in A, \quad \phi(f) \equiv f^p \mod p.
\]
If $A$ is a $\delta$-ring, then the map
\[
\phi : A \to A, \quad f \mapsto f^p + p \delta(f),
\]
obtained by composition of the section with the ghost map, is a Frobenius on $A$ and this construction is functorial.
Moreover, it is easy to verify that $\phi \circ \delta = \delta \circ \phi$.
Note that the multiplicative condition \eqref{deltmu} may then be rewritten in the asymmetric but sometimes more convenient way
\begin{equation} \label{asym}
\forall f,g \in A, \quad \delta(fg) = \delta(f) g^p + \phi(f)\delta(g) = f^p\delta(g) +\delta(f)\phi(g).
\end{equation}
Conversely, when $A$ is $p$-torsion-free, any Frobenius of $A$ comes from a unique $\delta$-structure through the formula
\[
\delta(f) = \frac {\phi(f) - f^p}p \in A.
\]

We will systematically use the fact that the category of $\delta$-rings has all limits and colimits and that they both preserve the underlying rings: actually the forgetful functor is conservative and has both an adjoint ($\delta$-envelope) and a coadjoint (Witt vectors).
We refer the reader to the article of James Borger and Ben Wieland \cite{BorgerWieland05} for a plethystic interpretation of these phenomena.

\begin{dfn}
Let $R$ be a $\delta$-ring.
If $A$ is an $R$-algebra, then its \emph{$\delta$-envelope} $A^\delta$ is a $\delta$-$R$-algebra which is universal for morphisms of $R$-algebras into $\delta$-$R$-algebras.
\end{dfn}

Actually, it follows from the above discussion that the forgetful functor from $\delta$-$R$-algebras to $R$-algebras has an adjoint $A \mapsto A^\delta$ (so that $\delta$-envelopes always exist and preserve colimits).
Be careful that, even if it does not show up in the notation, the $\delta$-envelope does depend on the base $\delta$-ring $R$.
However, if $R \to S$ is a morphism of $\delta$-rings, we have

\begin{equation} \label{bcdelt}
(S \otimes_R A)^\delta \simeq S \otimes_R A^\delta.
\end{equation}
Of course, it is always possible to define the $\delta$-envelope of a \emph{ring} $A$ by considering the case $R = \mathbb Z$.

\begin{xmp}
We have\footnote{For $R = \mathbb Z$, this is in fact the same as the free $\delta$-ring $\mathbb Z\{e\}$ of \cite{BhattScholze22}, lemma 2.11 or the ring of $p$-typical symmetric functions $\Lambda_{p}$ of \cite{BorgerWieland05}.} $R[x]^\delta = R[\{x_{i}\}_{i \in \mathbb N}]$ (polynomial ring with infinitely many variables) with the unique $\delta$-structure such that $\delta(x_{i}) = x_{i+1}$ for $i \in \mathbb N$.
More generally, since $\delta$-envelope preserves colimits, we have $ R[\{x_{k}\}_{k\in F}]^\delta = R[\{x_{k,i}\}_{k \in F, i \in \mathbb N}]$ with the unique $\delta$-structure such that $\delta(x_{k,i}) = x_{k,i+1}$ for $k \in F, i \in \mathbb N$.
\end{xmp}

\begin{dfn}
A \emph{$\delta$-ideal} in a $\delta$-ring $A$ is an ideal $I$ which is stable under $\delta$.
In general, the \emph{$\delta$-closure} $I_{\delta}$ of an ideal $I$ is the smallest $\delta$-ideal containing $I$.
\end{dfn}

Equivalently, a $\delta$-ideal is the kernel of a morphism of $\delta$-rings.
It immediately follows from formulas (\eqref{deltad}) and (\eqref{deltmu}) above that the condition for an ideal to be a $\delta$-ideal may be checked on generators.
As a consequence, if $I = \left(\left\{f_{i}\right\}_{i\in E}\right)$, then
\[
I_{\delta} = \left(\left\{\delta^j(f_{i})\right\}_{i\in E,j\in \mathbb N}\right).
\]
Also, if $R$ is a $\delta$-ring, $A$ an $R$-algebra and $I \subset A$ an ideal, then we have
\[
(A/I)^\delta = A^\delta/(IA^\delta)_{\delta}.
\]
This provides a convenient tool to describe $\delta$-envelopes by choosing a presentation:
\[
A = R[\{x_{k}\}_{k\in F}]/\left(\left\{f_{i}\right\}_{i\in E}\right) \Rightarrow A^\delta = R[\{x_{k,i}\}_{k \in F, i \in \mathbb N}]/\left(\left\{\delta^j(f_{i})\right\}_{i\in E,j\in \mathbb N}\right)
\]
(with $x_{k} \mapsto x_{k,0}$).

\begin{xmp}
We usually endow the polynomial ring $\mathbb Z[q]$ with the unique $\delta$-structure such that $\delta(q) = 0$.
Then the principal ideal generated by $q-1$ is a $\delta$-ideal:
since $q^p-1 \equiv (q-1)^{p}$ both modulo $p$ and modulo $q-1$, we can write $q^p-1 - (q-1)^{p} = p(q-1)c$ in the unique factorization domain $\mathbb Z[q]$, and then we have $\delta(q-1) = (q-1)c$.
\end{xmp}

We will need later the following property:

\begin{lem} \label{factdel}
Assume that $p$ lies in the Jacobson radical of both $A$ and $B$.
If $A \times B$ is endowed with a $\delta$-structure, then both $A$ and $B$ are $\delta$-ideals.
Equivalently, there exists a (unique) quotient $\delta$-structure on both $A$ and $B$.
\end{lem}

\begin{proof}
It is enough to prove that $A$ has a unique quotient $\delta$-structure, and we only need to check that $\delta(1, 0) = (0,0)$ because everything else is automatic.
If we write $\delta(1, 0) =: (f, g)$, then we have
\[
\delta(1,0) = \delta((1,0)^2) = 2(f,g)(1,0)^p+p(f,g)^2 = (2f+pf^2,pg^2)
\]
It follows that $f = 2f+pf^2$ and $g = pg^2$ or, in other words, that $(1+pf)f = 0$ and $(1-pg)g = 0$.
Since $p$ lies in the Jacobson radical of $A$ (resp. $B$) then $1 + pf$ (resp. $1-pg$) is invertible in $A$ (resp. B) and it follows that $f=0$ (resp. $g=0$).
\end{proof}

Note that the condition will be satisfied when $A$ and $B$ are both $(p)$-adically complete.
Note also that \emph{some} condition is necessary because the result does not hold for example when $A = B = \mathbb Q$.

It is important for us to recall also that a $p$-derivation $\delta$ on a ring $A$ is systematically $I$-adically continuous with respect to any finitely generated ideal $I$ containing $p$.
In particular, $\delta$ will then automatically extend in a unique way to a $\delta$-structure on the $I$-adic completion of $A$.

Finally, we will use the convenient terminology of $f$ being a \emph{rank one element} to mean $\delta(f) = 0$ (and consequently $\phi(f) = f^p$) and $f$ being a \emph{distinguished element} to mean $\delta(f) \in A^\times$ (and therefore $p \in (f^p, \phi(f))\subset (f, \phi(f))$).

\section{$\delta$-rings and twisted divided powers} \label{deltwist}

We fix a $\delta$-ring $R$ as well as a rank one element $q$ in $R$.
Alternatively, we may consider $R$ as a $\delta$-$\mathbb Z[q]$-algebra, where $q$ is then seen as a parameter with $\delta(q) = 0$, in which case we would still write $q$ instead of $q1_{R} \in R$. 
As a $\delta$-ring, $R$ is endowed with a Frobenius endomorphism $\phi$.
Note that $\phi(q) = q^p$ and the action on $q$-analogs of natural numbers\footnote{We write $(n)_{q} := \frac {q^n-1}{q-1}$.} is given by $\phi((n)_{q}) = (n)_{q^p}$.

We let $A$ be a $\delta$-$R$-algebra and we fix some rank one element $x$ in $A$.
The $R$-algebra $A$ is automatically endowed with a Frobenius $\phi$ which is semilinear (with respect to the Frobenius $\phi$ of $R$) and satisfies $\phi(x) = x^p$.
Although we could consider as well the \emph{relative} Frobenius, which is the $R$-linear map\footnote{We used to put a star and write $F_{A/R}^*$ in \cite{GrosLeStumQuiros19}.}
\[
F : A' := R {}_{{}_{\phi}\nwarrow}\!\!\otimes_{R} A \to A, \quad a {}_{{}_{\phi}\nwarrow}\!\!\otimes f \mapsto af^p + pa\delta(f),
\]
so that $\phi(f) = F(1 \otimes f)$, we will stick here to the \emph{absolute} Frobenius $\phi$ and modify our formulas from \cite{GrosLeStumQuiros19} accordingly.
Let us remark, however, that $A'$ has a natural $\delta$-structure and that $F$ is then a morphism of $\delta$-$R$-algebras.

There exists a unique structure of $\delta$-$A$-algebra on the polynomial ring $A[\xi]$ such that $x + \xi$ has rank one and we call it the \emph{symmetric} $\delta$-structure.
It is given by
\begin{equation} \label{xiform}
\delta(\xi) = \sum_{i=1}^{p-1} \frac 1p {p \choose i} x^{p-i}\xi^{i}
\end{equation}
(which depends on $x$ but not on $q$ or $\delta$).
Recall that we introduced in section 4 of \cite{LeStumQuiros15} the \emph{twisted powers}
\[
\xi^{(n)_{q}} := \xi(\xi + x - qx) \cdots (\xi + x - q^{n-1}x) \in A[\xi]
\]
(that clearly depend on both the choice of $q$ and $x$).
We will usually drop the index $q$ and simply write $\xi^{(n)}$.
They form an alternative basis for $A[\xi]$ (as an $A$-module) which is better adapted to working with $q$-analogs: if $\phi$ denotes the Frobenius of $A[\xi]$ attached to the symmetric $\delta$-structure, then we have the fundamental congruence
\begin{equation} \label{init}
\phi(\xi) \equiv \xi^{(p)} \mod (p)_{q}.
\end{equation}
This follows from corollary 7.6 of \cite{GrosLeStumQuiros19} and lemma 2.12 of \cite{LeStumQuiros15}, but may also be checked directly.
There exist more general explicit formulas for higher twisted powers:
we showed in proposition 7.5 of \cite{GrosLeStumQuiros19} that\footnote{In \cite{GrosLeStumQuiros19} we actually used $A_{n,i}$ and $B_{n,i}$ instead of $a_{n,i}$ and $b_{n,i}$.}
\[
\phi(\xi^{(n)}) = \sum_{i=n}^{pn} a_{n,i} x^{pn-i}\xi^{(i)}
\]
with
\begin{equation} \label{ani}
a_{n,i} := \sum_{j=0}^n (-1)^{n-j} q^{\frac {p(n-j)(n-j-1)}2} {n \choose j}_{q^p} {pj \choose i}_{q} \in \mathbb Z[q]
\end{equation}
(in which we use the $q$-analogs of the binomial coefficients, as in section 2 of \cite{LeStumQuiros15} for example).

In section 2 of \cite{GrosLeStumQuiros19}, we also introduced the ring of \emph{twisted divided polynomials} $A\langle\xi\rangle_{q}$, which depends on the choice\footnote{Actually, it depends on $q$ and $y := (1-q)x$.} of both $q$ and $x$ (but does not require any $\delta$-structure on $A$).
This is a commutative $A$-algebra.
As an $A$-module, it is free on some generators $\xi^{[n]_{q}}$ (called the \emph{twisted divided powers}) indexed by $n \in \mathbb N$.
The multiplication rule is quite involved:
\begin{equation} \label{mulrul}
\xi^{[n]_{q}} \xi^{[m]_{q}} = \sum_{0\leq i \leq n,m} q^{\frac{i(i-1)}2}{n + m -i \choose n}_{q}{n \choose i}_{q} (q-1)^{i}x^i\xi^{[n+m-i]_{q}}. 
\end{equation}
Again, we will usually drop the index $q$ and simply write $\xi^{[n]}$.
Heuristically\footnote{More precisely, there exists a unique natural homomorphism $A[\xi] \to A\langle\xi\rangle_{q}, \xi^{(n)} \mapsto (n)_{q}! \xi^{[n]}$.}, we have
\[
\forall n \in \mathbb N, \quad \xi^{[n]} = \frac {\xi^{(n)}}{(n)_{q}!}.
\]
In section 7 of \cite{GrosLeStumQuiros19}, we showed that there exists a \emph{divided Frobenius} map
\[
[F] : A'\langle \omega \rangle_{q^p} \to A \langle \xi \rangle_{q}
\] 
(we denote by $\omega$ the variable on the left hand side in order to avoid confusion) such that, heuristically again, 
\begin{equation} \label{frobdiv}
\forall n \in \mathbb N, \quad [F]\left(\omega^{[n]}\right) = \frac {\phi \left(\xi^{[n]}\right)} {(p)_{q}^n}.
\end{equation}
More precisely, we have the following formula
\begin{equation} \label{bcoef}
[F]\left(\omega^{[n]}\right) = \sum_{i=n}^{pn} b_{n,i} x^{pn-i}\xi^{[i]} 
\end{equation}
with
\[ 
b_{n,i} = \frac {(i)_{q}!}{(n)_{q^p}!(p)_{q}^n}a_{n,i} \in \mathbb Z[q]
\]
and $a_{n,i}$ as in \eqref{ani}.

We may then define the absolute frobenius of $A \langle \xi \rangle_{q}$ by composing the ``blowing up''
\[
A \langle \xi \rangle_{q} \to A'\langle \omega \rangle_{q^p}, \quad \xi^{[n]} \mapsto (p)_{q}^n\omega^{[n]}
\]
with the divided Frobenius $[F]$.
Note that this blowing up is a morphism of rings: this is actually a generic question that may be checked on polynomial rings.
One must also check that this is a lifting of the absolute frobenius.
We may clearly assume that $R = \mathbb Z[q]$ and $A = R[x]$ and we are reduced to showing that the map induced on $\mathbb Z[q,x]\langle \xi \rangle/p$ is the absolute frobenius.
Since this ring is $(n)_{q}$-torsion free for all $n > 0$, this boils down to the same assertion on $\mathbb F_{p}[q,x,\xi]$.
It is therefore sufficient to recall that $q$, $x$ and $x + \xi$ map respectively to $q^p$, $x^p$ and $(x+\xi)^p$.

When $A$ is $p$-torsion-free, the Frobenius of $A\langle\xi\rangle_{q}$ corresponds to a unique $\delta$-structure on $A\langle\xi\rangle_{q}$ and this provides a natural $\delta$-structure in general using the isomorphism

\begin{equation} \label{basch}
A\langle\xi\rangle_{q} \simeq A \otimes_{\mathbb Z[q,x]} \mathbb Z[q,x] \langle\xi\rangle_{q}
\end{equation}
(where $q$ and $x$ are seen as parameters).

\begin{rmk}
The rank one condition on $x$ is crucial for our results to hold.
Otherwise, the Frobenius of $A[\xi]$ will \emph{not} extend to $A\langle\xi\rangle_{q}$.
This is easily checked when $p = 2$.
Let us denote by $\delta_{c}$ the $\delta$ structure given by $\delta_{c}(x) = c \in R$ and by $\phi_{c}$ the corresponding Frobenius, so that
\[
\phi_{c}(x) = x^2 + 2c = \phi_{0}(x) + 2c \quad \mathrm{and} \quad \phi_{c}(\xi) = \xi^2 +2x\xi = \phi_{0}(\xi).
\]
Recall now from the first remark after definition 7.10 in \cite{GrosLeStumQuiros19} that
\[
\phi_{0}(\xi) = (1+q)(\xi^{[2]}+ x\xi).
\]
Then, the following computation shows that $\phi_{c}\left(\xi^{(2)}\right)$ is \emph{not} divisible by $(2)_{q^2} = 1 + q^2$ in general (so that $\phi_{c}\left(\xi^{[2]}\right)$ does \emph{not} exist in $A\langle\xi\rangle_{q}$):
\begin{align*}
\phi_{c}\left(\xi^{(2)}\right) & = \phi_{c}(\xi(\xi + (1-q)x)) \\
	& = \phi_{0}(\xi)(\phi_{0}(\xi) + (1-q^2)(\phi_{0}(x) + 2c)) \\
	& = \phi_{0}\left(\xi^{(2)}\right) + 2c(1-q^2)\phi_{0}(\xi) \\
	& = (1+q^2)\phi_{0}\left(\xi^{[2]}\right) + 2c(1-q)(1+q)^2(\xi^{[2]}+ x\xi).
\end{align*}
\end{rmk}

\section{$q$-divided powers and twisted divided powers}

As before, we let $R$ be a $\delta$-ring with fixed rank one element $q$.
We consider the maximal ideal $(p, q-1) \subset \mathbb Z[q]$ and we assume from now on that $R$ is actually a $\mathbb Z[q]_{(p,q-1)}$-algebra.

Note that, under this new hypothesis, we have $(k)_{q} \in R^\times$ as long as $p \nmid k$ so that $R/(p)_{q}$ is a $q$-divisible ring of $q$-characteristic $p$ in the sense of \cite{LeStumQuiros15}, which was a necessary condition for the main results of \cite{GrosLeStumQuiros19} to hold.

We have the following congruences when $k \in \mathbb N$:
\begin{equation} \label{cong}
(p)_{q^k} \equiv p \mod q-1 \quad \mathrm{and} \quad (p)_{q^k} \equiv (k)_{q}^{p-1}(q-1)^{p-1} \mod p,
\end{equation}
which both imply that $(p)_{q^k} \in (p, q-1)$.
It also follows that $(p)_{q^k}$ is a \emph{distinguished} element in $R$ since
\[
\delta((p)_{q^k}) \equiv \delta(p) = 1 - p^{p-1}\equiv 1 \mod (p, q-1)
\]
(the first congruence follows from the fact that $(q-1)$ is a $\delta$-ideal).

We recall now the following notion from \cite{BhattScholze22}, definition 16.2: 

\begin{dfn}
\begin{enumerate}
\item
If $B$ is a $\delta$-$R$-algebra and $J \subset B$ is an ideal, then $(B, J)$ is a \emph{$\delta$-pair}.
We may also say that the surjective map $B \twoheadrightarrow \overline B := B/J$ is a \emph{$\delta$-thickening}.
\item
If, moreover, $B$ is $(p)_{q}$-torsion-free and
\begin{equation}\label{dpcond}
\forall f \in J, \quad \phi(f) -(p)_{q}\delta(f) \in (p)_{q}J,
\end{equation}
then $(B,J)$ is a \emph{$q$-PD-pair}\footnote{There are a lot more requirements, that we may ignore at this moment, in definition 16.2 of \cite{BhattScholze22}.}.
We may also say that $J$ is a \emph{$q$-PD-ideal} or has \emph{$q$-divided powers} or that the map $B \twoheadrightarrow \overline B$ is a \emph{$q$-PD-thickening}.
\end{enumerate}
\end{dfn}

\begin{rmks}
\begin{enumerate}
\item Condition \eqref{dpcond} may be split in two, as Bhatt and Scholze do in loc.\ cit.: one may first require that $\phi(J) \subset (p)_{q}B$, 
then introduce the map
\[
\gamma : J \to B, f \mapsto \frac{\phi(f)}{(p)_{q}} - \delta(f)
\]
and also require that $\gamma(J) \subset J$.
\item In the special case where $J$ is a $\delta$-ideal, condition \eqref{dpcond} simply reads $\phi(J) \subset (p)_{q}J$, and this implies that $\phi^k(J) \subset (p^k)_{q}J$ for all $k \in \mathbb N$.
\item In general, the elements $f \in J$ that satisfy the property in condition \eqref{dpcond} form an ideal and the condition can therefore be checked on generators.
\end{enumerate}
\end{rmks}

\begin{xmps}
\begin{enumerate}
\item In the case $q =1$, condition \eqref{dpcond} simply reads
\[
\forall f \in J, \quad f^p \in pJ
\]
and $(B,J)$ is a $q$-PD-pair if and only if $B$ is $p$-torsion-free and $J$ has \emph{usual} divided powers (use lemma 2.35 of \cite{BhattScholze22}).
In other words, a $1$-PD-pair is a $p$-torsion-free PD-pair endowed with a lifting of Frobenius (the $\delta$-structure).
\item If $B$ is a $(p)_{q}$-torsion-free $\delta$-$R$-algebra, then the \emph{Nygaard ideal} $\mathcal N := \phi^{-1}((p)_{q}B)$ has $q$-divided powers (this is shown in \cite{BhattScholze22}, lemma 16.7).
This is the maximal $q$-PD-ideal in $B$.
Note that $\mathcal N$ is the first piece of the \emph{Nygaard filtration}.
\item If $B$ is a $(p)_{q}$-torsion-free $\delta$-$R$-algebra, then $J := (q-1)B$ has $q$-divided powers.
This is also shown in \cite{BhattScholze22}, but actually simply follows from the equalities
\[
\phi(q-1) =q^p-1 = (p)_{q}(q-1),
\]
since $(q-1)$ is a $\delta$-ideal.
\item
With the notations of the previous section, we may endow $A\langle\xi\rangle_{q}$ with the augmentation ideal $I^{[1]}$ generated by all the $\xi^{[n]}$ for $n \geq 1$.
When $A$ is $(p)_{q}$-torsion-free, this is a $q$-PD pair as the equality \eqref{frobdiv} shows (since $I^{[1]}$ is clearly a $\delta$-ideal).
\end{enumerate}
\end{xmps}

For later use, let us also mention the following:

\begin{lem} \label{formpair}
\begin{enumerate}
\item
If $(B, J)$ is a $q$-PD-pair, $B'$ is a $(p)_{q}$-torsion-free $\delta$-ring and $B \to B'$ is a morphism of $\delta$-rings, then $(B', JB')$ is a $q$-PD-pair.
\item
The category of $\delta$-pairs has all colimits and they are given by
\[
\varinjlim (B_{e}, J_{e}) = (B, J)
\]
with $B = \varinjlim B_{e}$ and $J = \sum J_{e}B$.
Colimits preserve $q$-PD-pairs as long as they are $(p)_{q}$-torsion free.
\end{enumerate}
\end{lem}

\begin{proof}
Both assertions concerning $q$-PD-pairs follow from the fact that the property in condition \eqref{dpcond} may be checked on generators.
The first part of the second assertion is a consequence of the fact that the category of $\delta$-rings has all colimits and that they preserve the underlying rings.
\end{proof}

\begin{rmks}
\begin{enumerate}
\item
As a consequence of the first assertion, we see that if $(B, J)$ is a $q$-PD-pair and $\mathfrak b \subset B$ is a $\delta$-ideal such that $B/\mathfrak b$ is $(p)_{q}$-torsion-free, then $(B/\mathfrak b, J + \mathfrak b/\mathfrak b)$ is also a $q$-PD-pair.
\item
As a particular case of the second assertion, we see that fibered coproducts in the category of $\delta$-pairs are given by
\[
(B_{1}, J_{1}) \otimes_{(B, J)} (B_{2}, J_{2}) = (B_{1} \otimes_{B} B_{2}, \mathrm{im}(B_{1} \otimes_{B} J_{2} + J_{1} \otimes_{B} B_{2})).
\]
Note that if $B_{1}$ is $B$-flat and $B_{2}$ is $(p)_{q}$ torsion-free, then $B_{1} \otimes_{B} B_{2}$ is automatically $(p)_{q}$-torsion-free.
\end{enumerate}
\end{rmks}

\begin{dfn}\label{q-PD-envelope}
Let $(B, J)$ be a $\delta$-pair.
Then (if it exists) its \emph{$q$-PD envelope} $\left(B^{[\ ]_{q}}, J^{[\ ]_{q}}\right)$ is a $q$-PD-pair that is universal for morphisms to $q$-PD-pairs:
there exists a morphism of $\delta$-pairs $(B, J) \to \left(B^{[\ ]_{q}}, J^{[\ ]_{q}}\right)$ such that any morphism $(B, J) \to (B', J')$ to a $q$-PD-pair extends uniquely to $\left(B^{[\ ]_{q}}, J^{[\ ]_{q}}\right)$.
\end{dfn}

\begin{rmks}
\begin{enumerate}
\item Note that the $q$-PD-envelope only depends on the underlying $\delta$-$\mathbb Z[q]$-algebra structure: we may as well take $R = B$.
\item In definition \ref{q-PD-envelope}, it might be necessary/useful for later developments to add the condition $\overline B \simeq \overline {B^{[\ ]_{q}}}$, but this is not clear yet.
\end{enumerate}
\end{rmks}

\begin{xmps}
\begin{enumerate}
\item When $q=1$, $B$ is $p$-torsion-free and $J$ is generated by a regular sequence modulo $p$, then the $q$-PD-envelope of $B$ is its \emph{usual} PD-envelope (corollary 2.39 of \cite{BhattScholze22}).
\item If $B$ is $(p)_{q}$-torsion-free and $J$ already has $q$-divided powers, then $B^{[\ ]_{q}} = B$.
\item \label{J=B} If $J= B$, then $B^{[\ ]_{q}} = B[\{1/(n)_{q}\}_{n \in \mathbb N}]$ ($\neq 0$ in general).
\end{enumerate}
\end{xmps}

As a consequence of lemma \ref{formpair}, let us mention the following:

\begin{lem} \label{paircons}
\begin{enumerate}
\item
Let $(B, J)$ be a $\delta$-pair and $\mathfrak b \subset B$ a $\delta$-ideal.
Assume that $(B, J)$ has a $q$-PD-envelope and $B^{[\ ]_{q}}/\mathfrak bB^{[\ ]_{q}}$ is $(p)_{q}$-torsion-free.
Then, 
\[
\left(B^{[\ ]_{q}}/\mathfrak bB^{[\ ]_{q}}, (J^{[\ ]_{q}}+\mathfrak bB^{[\ ]_{q}})/\mathfrak bB^{[\ ]_{q}}\right)
\]
is the $q$-PD-envelope of $(B/\mathfrak b, J+\mathfrak b/\mathfrak b)$.
\item
Let $\{(B_{e}, J_{e})\}_{e\in E}$ be a commutative diagram of $\delta$-pairs all having a $q$-PD-envelope, $B := \varinjlim B_{e}$ and $J = \sum J_{e}B$.
If $B':=\varinjlim B_{e}^{[\ ]_{q}}$ is $(p)_{q}$-torsion free, then $\left(B', \sum J_{e}^{[\ ]_{q}}B'\right)$ is the $q$-PD-envelope of $(B,J)$.
\end{enumerate}
\end{lem}

\begin{proof}
Concerning both assertions, we already know from lemma \ref{formpair} (and the remarks thereafter) that the given pair is a $q$-PD-pair and the universal property follows from the universal properties of quotient and colimit respectively.
\end{proof}

The next example is fundamental for us (we use the notations of the previous section):

\begin{thm} \label{algun}
If $A$ is a $(p)_{q}$-torsion-free $\delta$-$R$-algebra with fixed rank one element $x$ and $A[\xi]$ is endowed with the symmetric $\delta$-structure, then $(A\langle\xi\rangle_{q}, I^{[1]})$ is the $q$-PD-envelope of $(A[\xi], (\xi))$.
Moreover, we have an isomorphism of $\delta$-$R$-algebras
\[
A\langle\xi\rangle_{q} \simeq A[\xi][\phi(\xi)/(p)_q]^\delta
\]
with the $\delta$-envelope of $A[\xi][\phi(\xi)/(p)_q]$ over $A[\xi]$.
\end{thm}

\begin{proof}
We already know from \eqref{basch} that $A\langle\xi\rangle_{q}$ is stable under the base change of $A$ and it follows from \eqref{bcdelt} that the same holds for $A[\xi][\phi(\xi)/(p)_q]^\delta$.
We may therefore assume that $R = \mathbb Z[q]_{(p,q-1)}$ and $A =R[x]$ and we first construct some other basis of $A\langle\xi\rangle_{q}$ as an $A$-module that will be useful later.
We endow $A\langle\xi\rangle_{q}$ with the \emph{degree filtration} by the $A$-submodules $F_{n}$ generated by $\xi^{[k]}$ for $k \leq n$.
Formula \eqref{bcoef} shows that Frobenius sends $F_{n}$ into $F_{pn}$.
Since this is clearly also true for the $p$th power map and $A$ is $p$-torsion-free, we see that the same also holds for $\delta$ (recall that $\delta(u) = \frac 1p(\phi(u) -u^p)$ and for $\gamma$ (recall that $\gamma(u) = \frac 1{(p)_{q}}\phi(u) - \delta(u)$).
At this point, we may actually notice that, although $\delta$ is not an additive map, formula \eqref{deltad} shows that it satisfies
\begin{equation} \label{delteq}
\forall u, v \in F_{n}, \quad (u \equiv v \mod F_{n-1}) \Rightarrow (\delta (u) \equiv \delta(v) \mod F_{pn-1}),
\end{equation}
and the analogous property is also satisfied by $\gamma$ when $u,v \in I^{[1]}$.
Now, it is not difficult to compute the following leading terms
\[
\phi\left(\xi^{[n]}\right) \equiv \frac {(np)_{q}!}{(n)_{q^p}!} \xi^{[np]} \mod F_{np-1}
\]
and
\[
\left(\xi^{[n]}\right)^p \equiv \frac {(np)_{q}!}{((n)_{q}!)^p} \xi^{[np]} \mod F_{np-1}.
\]
It follows that
\begin{equation}\label{delta}
\delta\left(\xi^{[n]}\right) \equiv d_{n}\xi^{[np]} \mod F_{np-1} \quad \mathrm{with}\quad d_{n} = \frac {(((n)_{q}!)^p-(n)_{q^p}!)(np)_{q}!}{p(n)_{q^p}!((n)_{q}!)^p} \in R.
\end{equation}
We claim that $d_{p^r} \in R^\times$ when $r \in \mathbb N \setminus \{0\}$.
Since $R$ is a local ring and $q-1$ belongs to the maximal ideal, we may assume that $q = 1$ and in this case,
\[
d_{p^r} = \frac {((p^r!)^p- p^r!)(p^{r+1}!)}{pp^r!(p^r!)^p} = \frac {((p^{r}!)^{p-1}- 1)(p^{r+1}!)}{p(p^r!)^p}.
\]
We have
\begin{equation} \label{compvp}
v_{p}(p^{r+1}!) - pv_{p}(p^r!) - 1 = \frac {p^{r+1}-1}{p-1} - p \frac {p^r-1}{p-1} - 1 = 0
\end{equation}
and therefore $v_{p}(d_{p^r}) = 0$.
We can now show by induction that 
\[
\forall r \in \mathbb N, \exists c_{r} \in R^\times, \quad \gamma^r(\xi)\equiv c_{r}\xi^{[p^{r}]} \mod F_{p^r-1}.
\]
The formula trivially holds for $r=0$.
Moreover, since $d_1=0$ in \eqref{delta}, we have $\delta(\xi) \equiv 0 \mod F_{p-1}$ and therefore
\[
\gamma(\xi) \equiv \frac {\phi(\xi)}{(p)_q} \equiv (p-1)_q! \xi^{[p]} \mod F_{p-1}.
\]
Thus the formula also holds for $r=1$ and, assuming that it holds for some $r \geq 1$, we will have, thanks to property \eqref{delteq} and the asymmetric formula \eqref{asym},
\begin{align*}
\gamma^{r+1}(\xi) & \equiv \gamma\left(c_{r}\xi^{[p^{r}]}\right) \mod F_{p^{r+1}-1} \\
& \equiv \frac 1{(p)_{q}}\phi\left(c_{r}\xi^{[p^{r}]}\right) - \delta\left(c_{r}\xi^{[p^{r}]}\right) \mod F_{p^{r+1}-1} \\
& \equiv \frac 1{(p)_{q}}\phi(c_{r})\phi\left(\xi^{[p^{r}]}\right) - \delta(c_{r})\phi\left(\xi^{[p^r]}\right) - c_{r}^p\delta\left(\xi^{[p^r]}\right) \mod F_{p^{r+1}-1}\\
& \equiv \left((p)_{q}^{p^{r}-1}(\phi(c_{r}) - (p)_{q}\delta(c_{r})) b_{p^r,p^{r+1}} -c_{r}^pd_{p^r}\right) \xi^{[p^{r+1}]} \mod F_{p^{r+1}-1}
\end{align*}
with $b_{n,i} \in R$ as in \eqref{bcoef}.
Since $c_{r}^pd_{p^r} \in R^\times$, it is then sufficient to recall that $(p)_{q}$ belongs to the maximal ideal $(p, q-1)$ and the coefficient is therefore necessarily invertible, as asserted.

Now, if $n \in \mathbb N$ has $p$-adic expansion $n = \sum_{r \geq 0} k_{r} p^r$, then
\[
v_{n} := \prod_{r\geq0} \gamma^r(\xi)^{k_{r}} \equiv \prod_{r\geq0} c_{r}^{k_{r}}\prod_{r\geq0} (\xi^{[p^{r}]})^{k_{r}} \equiv c(n)\xi^{[n]} \mod F_{n-1}
\]
with $c(n) \in R^\times$.
It follows that $\{v_{n}\}_{n \in \mathbb N}$ is a basis for the free $A$-module $A\langle \xi \rangle_{q}$.
At this point, we should recall that $I^{[1]}$ is a $q$-PD-ideal because it is a $\delta$-ideal and $\phi(I^{[1]}) \subset (p)_{q} A\langle \xi \rangle_{q}$.
In particular, $I^{[1]}$ is stable under $\gamma$.
It follows that $\{v_{n}\}_{n > 0}$ is a basis for the free $A$-module $I^{[1]}$.
One can also show, exactly as we did above (but with a shift in the indices) that
\begin{equation} \label{newb2}
\forall r \in \mathbb N, \exists c'_r \in R^\times, \quad \delta^{r}\left(\frac{\phi(\xi)}{(p)_q}\right) \equiv c'_{r}\xi^{[p^{r+1}]} \mod F_{p^{r+1}-1}.
\end{equation}
Then, if $n$ has $p$-adic expansion $n = \sum_{r \geq 0} k_{r} p^r$ 	and we set
\[
v'_{n} := \xi^{k_0}\prod_{r\geq0} \delta^r(\phi(\xi)/(p)_q)^{k_{r+1}},
\]
we obtain another basis $\{v'_{n}\}_{n \in \mathbb N}$ for the free $A$-module $A\langle \xi \rangle_{q}$.

Assume now that $(B,J)$ is a $q$-PD-pair and that we are given a morphism of $\delta$-pairs $u : (A[\xi], (\xi)) \to (B,J)$.
Then $g := u(\xi) \in J$ and therefore $\phi(g)/(p)_q \in B$.
It follows that $u$ extends uniquely to a morphism of $R$-algebras
\[
A[\xi][\phi(\xi)/(p)_q] \simeq A[\xi,z]/((p)_qz - \phi(\xi)) \to B,
\]
and since $B$ is a $\delta$-$R$-algebra, this morphism then extends uniquely to a morphism of $\delta$-$R$-algebras
\[
U : A[\xi][\phi(\xi)/(p)_q]^\delta \simeq A[\xi][z]^\delta/((p)_qz - \phi(\xi))_\delta \to B
\]
(the $\delta$-envelopes are taken over $A[\xi]$, which is already a $\delta$-ring).
As a particular case, we can apply this construction to the canonical map
\[
\lambda : (A[\xi], (\xi)) \to \left( A\langle \xi \rangle_{q}, I^{[1]}\right),
\]
and we get a morphism of $\delta$-$R$-algebras
\[
\Lambda : A[\xi][\phi(\xi)/(p)_q]^\delta \to A\langle \xi \rangle_{q}.
\]
We will show below that $\Lambda$ is an isomorphism.
As a consequence, the map $u : A[\xi] \to B$ will extend uniquely to a morphism of $\delta$-rings that we may still call $U : A\langle \xi \rangle_q \to B$.
In particular, $U$ commutes with $\gamma$.
Since $\{v_{n}\}_{n > 0}$ is a basis for the free $A$-module $I^{[1]}$, it follows that $U$ sends $I^{[1]}$ into $J$.
In other words, $U$ is a morphism of $\delta$-pairs that extends uniquely our original morphism $u$.
This shows that $(A\langle\xi\rangle_{q}, I^{[1]})$ has the expected universal property.

It remains to show that $\Lambda$ is an isomorphism. The $\delta$-envelope $A[\xi][z]^\delta$ (over $A[\xi]$) of the polynomial ring $A[\xi, z]$ is the polynomial ring over $A$ on infinitely many variables $\{z_r\}_{r \in \mathbb N}$ where
\[
z_0 := \xi \quad \mathrm{and} \quad z_r := \delta^{r-1}(z)\ \mathrm{if}\ r>0.
\]
In particular, this is a free $A$-module with basis $\underline z^{\underline k} := \prod_{r \geq 0} z_r^{k_r}$ where $k_r = 0$ for $r >> 0$.
We will endow this free $A$-module with the \emph{weighted degree} defined by $\deg(\underline z^{\underline k}) = \sum_{r \geq 0} k_rp^r$ (and $\deg(0) = - \infty$).
In order to lighten the notations (and only in this proof), we will set
\[
E := A[\xi][\phi(\xi)/(p)_q]^\delta.
\]
We define the degree of $u \in E$ as the minimum (weighted) degree of all its representations in $A[\xi][z]^\delta$ so that
\[
\forall u,v \in E, \quad \deg(u+v) \leq \max\{\deg(u), \deg(v)\} \quad \mathrm{and} \quad \deg(uv) \leq \deg(u) + \deg(v).
\]
Using the multiplicative formula
\[
\forall u,v \in E, \quad \delta(uv) =u^p\delta(v) + \delta(u)v^p +p\delta(u)\delta(v),
\]
one easily checks by induction on $n$ that
\begin{equation} \label{muldeg}
\forall u \in E, \quad \deg(u) < n \Rightarrow \deg(\delta(u)) < pn.
\end{equation}
Note that we have an analogous property for $u \mapsto u^p$, and therefore also for the map $u \mapsto \phi(u)$.
We consider now the filtration of $E$ by the $A$-submodules $E_n$ defined by the condition $\deg (u) \leq n$ for $n \in \mathbb N$.
Let us then improve on \eqref{muldeg} and show that
\begin{equation} \label{deltaeq}
\forall u, v \in E_{n}, \quad (u \equiv v \mod E_{n-1}) \Rightarrow (\delta (u) \equiv \delta(v) \mod E_{pn-1}).
\end{equation}
If we write $v = u + w$ with $\deg(w) < n$, we have
\[
\delta(v) = \delta(u) + \delta(w) - \sum_{k=1}^{p-1} \frac 1p {p \choose k} u^{p-k}w^{k}.
\]
It follows from \eqref{muldeg} that $\deg(\delta(w)) < pn$ and we also have for all $0 < k < p$,
\begin{align*}
\deg (u^{p-k}w^{k}) &\leq (p-k)\deg(u)+ k \deg(w) 
\\& \leq (p-k)n+ k(n-1)
\\& = pn-k
\\& \leq pn-1.
\end{align*}
We will now prove by induction that, if we denote by $\overline z_r$ the class of $z_r$ in $E$, then
\begin{equation} \label{invert}
\forall r \geq 0, \exists e_r \in R^\times \quad e_r{\overline z}_{r}^p \equiv (p)_{q^{p^{r}}}{\overline z}_{r+1} \mod E_{p^{r+1}-1}.
\end{equation}
The formula holds when $r=0$ with $e_0=1$ because $\phi(\xi) = (p)_q\overline z$ and $\phi(\xi) \equiv \xi^p \mod E_{p-1}$ thanks to equation \eqref{xiform} which shows that $\deg(\delta(\xi)) < p$.
For the induction process, we also need the case $r=1$ and we use the formula $\delta(\phi(\xi)) = \delta((p)_q\overline z)$.
On the one hand, we have
\[
\delta(\phi(\xi)) = \phi(\delta(\xi)) \equiv 0 \mod E_{p^2-1}
\]
thanks to property \eqref{muldeg} (or more precisely its $\phi$ analog).
On the other hand,
\[
\delta((p)_q\overline z) = \delta((p)_q)\overline z^p + (p)_{q^p}\delta(\overline z).
\]
Property \eqref{invert} therefore holds when $r=1$ with $e_1=-\delta((p)_q)$ wich is invertible because $(p)_q$ is distinguished.
Now we fix $r \geq 2$ and we assume that for all $s < r$, there exists $e_{s} \in R^\times$ such that
\begin{equation} \label{indhyp}
 e_{s}{\overline z}_{s}^p \equiv (p)_{q^{p^{s}}}{\overline z}_{s+1} \mod E_{p^{s+1}-1}.
\end{equation}
From the case $s=r-1$ and property \eqref{deltaeq}, we get
 \[
 \delta(e_{r-1}{\overline z}_{r-1}^p) \equiv \delta((p)_{q^{p^{r-1}}}{\overline z}_{r}) \mod E_{p^{r+1}-1},
\]
which can be rewritten as
\[
\delta(e_{r-1}){\overline z}_{r-1}^{p^2} + \phi(e_{r-1})\delta({\overline z}_{r-1}^p) \equiv \delta((p)_{q^{p^{r-1}}}) {\overline z}_r^p + (p)_{q^{p^{r}}} {\overline z}_{r+1} \mod E_{p^{r+1}-1},
\]
or, for later use,
\begin{equation}\label{lateruse}
\delta(e_{r-1}){\overline z}_{r-1}^{p^2} + \phi(e_{r-1})\delta({\overline z}_{r-1}^p) - \delta((p)_{q^{p^{r-1}}}) {\overline z}_r^p \equiv (p)_{q^{p^{r}}} {\overline z}_{r+1} \mod E_{p^{r+1}-1}.
\end{equation}
At this point, it is necessary to introduce the $R$-submodule $W_{r+1}$ of $E$ generated by all $\underline {\overline z}^{\underline k}$ such that
\begin{equation} \label{wcond}
\deg(\underline {z}^{\underline k}) \leq p^{r+1} \quad \mathrm{and} \quad k_{r+1}=0
\end{equation}
(so that $z_{r+1}$ is ruled out).
We will now show that
\begin{equation} \label{wind}
\forall u \in W_{r+1}, \exists e \in R, \quad u \equiv e{\overline z}_{r}^p \mod E_{p^{r+1}-1}.
\end{equation}
We may clearly assume that $u = \overline {\underline z}^{\underline k}$ and that conditions \eqref{wcond} are satisfied.
If $\deg(u) < p^{r+1}$, then we are done (put $e=0$).
Otherwise, since $k_{r+1} = 0$, there exists $s\in \{0, \ldots, r\}$ such that $k_s \geq p$.
We can then write $u=v\overline z_s^p$ with $v = \overline {\underline z}^{\underline l}$ ($l_s = k_s-p$ and $l_r=k_r$ otherwise) and proceed by infinite descent on $|k| := \sum_{i=0}^rk_i \geq 0$.
In the case $s=r$, we have $v=1$ and we are done.
So we may assume now that $s < r$ and use our induction hypothesis \eqref{indhyp} so that
\[
u \equiv\frac{(p)_{q^{p^{s}}}}{e_s}v{\overline z}_{s+1} \mod E_{p^{r+1}-1}
\] 
and $v z_{s+1}= \underline z^{\underline m}$ ($m_{s+1} = 1$ and $m_r=l_r$ otherwise) with $|\underline m| < |\underline k|$.
Property \eqref{wind} is now proved and we can apply it to
\[
u := \delta(e_{r-1}){\overline z}_{r-1}^{p^2} + \phi(e_{r-1})\delta({\overline z}_{r-1}^p) - \delta((p)_{q^{p^{r-1}}}) {\overline z}_r^p.
\]
Using \eqref{lateruse}, we obtain formula \eqref{invert} but we still need to prove that $e_r$ is invertible.

In order to show that, we can assume as above that $q=1$.
Then, we can use corollary 2.39 of \cite{BhattScholze22} and identify $E$ with the the divided polynomial ring $ A \langle \xi \rangle$.
Using formula \eqref{newb2}, we also see that $E_n$ then identifies with $F_n$.
In this new setting, since we already showed property \eqref{invert} up to the fact that $e_r$ is invertible, we can write
\[
e_r\left(\delta^{r-1}\left(\frac{\phi(\xi)}{p}\right)\right)^p \equiv p\delta^{r}\left(\frac{\phi(\xi)}{p}\right) \mod F_{p^{r+1}-1}.
\]
But we know from \eqref{newb2} that there exists $c'_r, c'_{r-1} \in R^\times$ such that
\[
\delta^{r}\left(\frac{\phi(\xi)}{p}\right) \equiv c'_{r}\xi^{[p^{r+1}]} \quad \mathrm{and} \quad \left(\delta^{r-1}\left(\frac{\phi(\xi)}{p}\right)\right)^p \equiv c_{r-1}'^p \left(\xi^{[p^{r}]}\right)^p \mod F_{p^{r+1}-1}.
\]
Moreover, it follows from computation \eqref{compvp} that
\[
\exists f\in R^\times, \quad \left(\xi^{[p^{r}]}\right)^p \equiv fp \xi^{[p^{r+1}]} \mod F_{p^{r+1}-1}.
\]
Therefore, since $R$ is $p$-torsion free, $e_r=f^{-1}c’_r c’^{-p}_{r-1} \in R^\times$ and property \eqref{invert} is now completely proved.

In the same way as we showed property \eqref{wind}, one can now prove that $E$ is generated as an $A$-module by $\{\underline {\overline z}^{\underline k}, \quad k_r < p \}$.
In other words, we want to show that any $u \in E$ is an $A$-linear combination of $\underline{\overline z}^{\underline{k}}$ with all $k_r<p$.
We proceed by induction on $n := \deg(u)$, the result being trivial for $n<p$.
We may clearly assume that $u=\underline{\overline z}^{\underline{k}}$ with $\sum_r k_r p^r \leq n$.
Now, we proceed by infinite descent on $|\underline k| := \sum_{i \geq 0} k_i \geq 0$.
Assume that some $k_s\ge p$ and write $u=v \overline z_s^p$, where $v=\underline{z}^{\underline{l}}$ with $l_s=k_s-p$ and $l_r=k_r$ for $r\neq s$.
Thanks to \eqref{indhyp}, we have $z_s^p \equiv a z_{s+1} \mod E_{p^{s+1}-1}$ with $a\in R$.
Since $\deg(v) \leq n - p^{s+1}$, it follows that $u \equiv a v z_{s+1} \mod E_{n-1}$.
Since $v z_{s+1} \in E_n$ and $v z_{s+1} = \underline z^{\underline m}$ with $|\underline m| < |\underline k|$, we are done.

Let us come back to our morphism
\[
\Lambda : E \to A\langle \xi \rangle_q.
\]
We can define an $A$-linear section of $\Lambda$ by sending the basis vector $v'_n$ to $\underline {\overline z}^{\underline k}$ if $n$ has $p$-adic expansion $n = \sum_{r \geq 0} k_{r} p^r$.
Since $E$ is generated as an $A$-module by $\{\underline {\overline z}^{\underline k}, \quad k_r < p \}$, this section is surjective so that $\Lambda$ is bijective and we can identify both $\delta$-$R$-algebras as we claimed.
\end{proof}

\begin{rmks}
\begin{enumerate}
\item
In the proof of theorem \ref{algun}, it is actually not obvious at all from the explicit formulas that that the various coefficients are indeed invertible.
For example, in the simplest non trivial case $p=2$ and $r = 1$, we have $d_{2} = q + q^2 + q^3$, which is a non-trivial unit.
\item Be careful that the isomorphism in theorem \ref{algun} is very specific to our situation and it is not true for example that, if $X$ is an indeterminate, then $R[X, \phi(X)/(p)_q]^\delta$ is the $q$-PD-envelope of $R[X]^\delta$ with respect to the augmentation ideal unless $q=1$.
\item 
Our result may be seen as a decompletion in our situation of lemma 16.10 of \cite{BhattScholze22}.
Note however that it is not even clear \emph{a priori} that $A[\xi][\phi(\xi)/(p)_q]^\delta$ is $(p)_q$-torsion free before completion.
\item
Theorem \ref{algun} is closely related to Jonathan Pridham's work in \cite{Pridham19}.
For example, his lemma 1.3 shows that (when $A$ is not merely a $\delta$-ring but actually a $\lambda$-ring)
\[
\lambda^n\left(\frac \xi{q-1}\right) = \frac 1{(q-1)^n} \xi^{[n]} \quad \left( = \left(\frac \xi{q-1}\right)^{[n]}\right)
\]
when $q-1 \in R^\times$ (in which case $A[\xi] = A\langle \xi \rangle_{q}$).
\item
Note also that, as a corollary of our theorem, one recovers the existence of the $q$-logarithm (after completion, but see below) as in lemma 2.2.2 of \cite{AnschuetzLeBras19b}.
\end{enumerate}
\end{rmks}

\section{Complete $q$-PD-envelopes}

As before, $R$ denotes a $\delta$-ring with fixed rank one element $q$ and we assume that $R$ is actually a $\mathbb Z[q]_{(p,q-1)}$-algebra.
Any $R$-algebra $B$ will be implicitly endowed with its $(p, q-1)$-adic topology and we will denote by $\widehat B$ or $B^\wedge$ its completion for this topology, which is automatically a $\mathbb Z_{p}[[q-1]]$-algebra (as in \cite{BhattScholze22}).
We recall that, if $B$ is a $\delta$-$R$-algebra, then $\delta$ is necessarily continuous and extends therefore uniquely to $\widehat B$.
We also want to mention that a complete ring is automatically $(p)_{q}$-complete. Actually, the $(p, q-1)$-adic topology and the $(p, (p)_{q})$-adic topology coincide thanks to congruences \eqref{cong}, and we will often simply call this the adic topology.

\begin{rmk}
Besides the usual completion $\widehat B$ of $B$, one may also consider its \emph{derived completion} (\cite[\href{https://stacks.math.columbia.edu/tag/091N}{Tag 091N}]{stacks-project}) $R\varprojlim B^\bullet_{n}$, where $B^\bullet_{n}$ denotes the Koszul complex
\[
\xymatrix@R=0cm{B \ar[r] & B \oplus B \ar[r] & B \\ f \ar@{|->}[r] & (p^nf,-(q-1)^nf) \\ & (f, g) \ar@{|->}[r] & (q-1)^nf+p^ng.}
\]
More generally, one can consider the derived completion functor
\[
K^\bullet \mapsto \widehat {K^\bullet} := R\varprojlim (K^\bullet \otimes^L_{B} B^\bullet_{n})
\]
on complexes of $B$-modules.
If $M$ is a $B$-module and $M[0]$ denotes the corresponding complex concentrated in degree zero, then there exists a canonical map $\widehat {M[0]} \to \widehat M[0]$ which is \emph{not} an isomorphism in general (but see below).
\end{rmk}

We recall that an abelian group $M$ has \emph{bounded $p^\infty$-torsion} if
\[
\exists l \in \mathbb N, \forall s \in M, \forall m \in \mathbb N, \quad p^ms = 0 \Rightarrow p^ls=0.
\]
We will need the following result:

\begin{lem} \label{bdcomp}
If $M$ has bounded $p^\infty$-torsion, so does its $(p)$-adic completion.
\end{lem}

\begin{proof}
We use the notations of the definition and we assume that $\{s_{n}\}_{n \in \mathbb N}$ is a sequence in $M$ such that $p^ms_{n} \to 0$ when $n \to \infty$ for some $m \geq l$.
Thus, given $k \in \mathbb N$ with $k \geq l$, if we write $k' = k+m-l$, there exists $N \in \mathbb N$ such that for all $n \geq N$, we can write $p^ms_{n} = p^{k'}t$ with $t \in M$.
But then $p^m(s_{n} - p^{k'-m}t) = 0$ and therefore already $p^l(s_{n} - p^{k'-m}t) = 0$ or $p^ls_{n} = p^{k}t$.
This shows that $p^ls_{n} \to 0$.
\end{proof}

\begin{dfn}
An $R$-algebra $B$ is \emph{bounded}\footnote{We should perhaps say \emph{$q$-bounded} or even \emph{$p$-$q$-bounded}, but we prefer the simpler name.} if $B$ is $(p)_{q}$-torsion free and $B/(p)_{q}$ has bounded $p^\infty$-torsion.
\end{dfn}

\begin{rmks}
\begin{enumerate}
\item
If $B$ is bounded, then the computations carried out in the proof of lemma 3.7 of \cite{BhattScholze22} show that $\widehat {B[0]} = \widehat B[0]$.
In other words, for our purpose, there is no reason in this situation to introduce the notion of derived completion.
\item
It is equivalent to say that $B$ is a complete bounded $\delta$-ring or that $(B, (p)_{q})$ is a \emph{bounded prism} in the sense of \cite{BhattScholze22}, definition 1.4.
\end{enumerate}
\end{rmks}

\begin{prop} \label{topind}
If $B$ is a bounded $R$-algebra, then the adic topology on $(p)_{q}B$ is identical to the topology induced by the adic topology of $B$.
\end{prop}

\begin{proof}
Since the topology is defined by the ideal $(p, (p)_{q})$, it is actually sufficient to show that the $(p)$-adic topology on $(p)_{q}B$ is identical to the topology induced by the $(p)$-adic topology of $B$:
\[
\forall n \in \mathbb N, \exists m \in \mathbb N, \quad (p)_{q}B \cap p^mB \subset p^n(p)_{q}B.
\]
In other words, we have to show that
\[
\forall n \in \mathbb N, \exists m \in \mathbb N, \forall f \in B, \quad (\exists g \in B, p^mf = (p)_{q}g) \Rightarrow (\exists h \in B, p^mf = (p)_{q}p^nh).
\]
But since $B/(p)_{q}$ has bounded $p^\infty$-torsion, we know that
\[
\exists l \in \mathbb N, \forall f \in B, \forall m \in \mathbb N, \quad (\exists g \in B, p^mf = (p)_{q}g) \Rightarrow (\exists h \in B, p^lf = (p)_{q}h).
\]
It is therefore sufficient to set $m = n+l$.
\end{proof}

As a consequence, we see that the $q$-PD-condition is closed:

\begin{cor}\label{closcond}
If $B$ is a bounded $\delta$-$R$-algebra and $J \subset B$ is a $q$-PD- ideal, then its closure $\overline J^{\mathrm {cl}}$ for the adic topology is also a $q$-PD-ideal.
\end{cor}

\begin{proof}
Recall that the condition means that $(\phi -(p)_{q}\delta) (J) \subset (p)_{q}J$.
By continuity, we will have $(\phi -(p)_{q}\delta) (\overline J^{\mathrm {cl}}) \subset \overline {(p)_{q}J}^{\mathrm {cl}}$.
Now, it follows from proposition \ref{topind} that multiplication by $(p)_{q}$ induces a homeomorphism $B \simeq (p)_{q}B$ and therefore $(p)_{q}\overline {J}^{\mathrm {cl}} =\overline {(p)_{q}J}^{\mathrm {cl}}$.
\end{proof}

We also obtain that boundedness is preserved by completion:

\begin{cor} \label{compbound}
If $B$ is a bounded $R$-algebra, then $\widehat B$ also is bounded and the ideal $(p)_{q}\widehat B$ is closed in $\widehat B$ (or equivalently $\widehat B/(p)_{q}\widehat B$ is complete).
\end{cor}

\begin{proof}
It follows from proposition \ref{topind} that the sequence
\[
0 \longrightarrow B \overset {(p)_{q}} \longrightarrow B \longrightarrow B/(p)_{q} \longrightarrow 0
\]
is \emph{strict} exact, and then the sequence
\[
0 \longrightarrow \widehat B \overset {(p)_{q}} \longrightarrow \widehat B \longrightarrow \widehat {B/(p)_{q}} \longrightarrow 0
\]
 is therefore also strict exact.
In particular, $\widehat B$ is $(p)_{q}$-torsion free and the ideal $(p)_{q}\widehat B$ is closed in $\widehat B$.
Moreover, $\widehat B/(p)_{q} = \widehat {B/(p)_{q}} $ has bounded $p^\infty$-torsion thanks to lemma \ref{bdcomp}.
\end{proof}

The next result is basic but very useful in practice:

\begin{lem} \label{bounded}
If $R$ is bounded and $B$ is flat (over $R$) then $B$ is also bounded.
\end{lem}

\begin{proof}
Upon tensoring with $B$ over $R$, the exact sequence
\[
0 \longrightarrow R \overset {(p)_{q}} \longrightarrow R \longrightarrow R/(p)_{q} \longrightarrow 0
\]
provides an exact sequence
\[
0 \longrightarrow B \overset {(p)_{q}} \longrightarrow B \longrightarrow B/(p)_{q}\longrightarrow 0.
\]
This shows that $B$ is $(p)_{q}$-torsion free.
Also, upon tensoring with $B$ over $R$, the exact sequence
\[
0 \longrightarrow R/(p)_{q}[p^n] \longrightarrow R/(p)_{q} \overset {p^n}\longrightarrow R/(p)_{q} \longrightarrow R/((p)_{q}, p^n) \longrightarrow 0
\]
provides another exact sequence
\[
0 \longrightarrow B\otimes_{R} (R/(p)_{q})[p^n] \longrightarrow B/(p)_{q} \overset {p^n}\longrightarrow B/(p)_{q} \longrightarrow B/((p)_{q}, p^n) \longrightarrow 0.
\]
This shows that $B\otimes_{R} (R/(p)_{q})[p^n] = (B/(p)_{q})[p^n]$.
It follows that $B/(p)_{q}$ has bounded $p^\infty$-torsion (with the same bound as $R/(p)_{q}$).
\end{proof}

\begin{dfn}
A $\delta$-pair $(B, J)$ is \emph{complete} if $B$ is complete and $J$ is closed (or equivalently $\overline B := B/J$ is also complete).
More generally, the \emph{completion} of a $\delta$-pair $(B, J)$ is the $\delta$-pair $(\widehat B, \overline J^{\mathrm {cl}})$ where $\overline J^{\mathrm {cl}}$ denotes the closure of $J\widehat B$ in $\widehat B$.
\end{dfn}

\begin{rmks}
\begin{enumerate}
\item
We usually have $J\widehat B \neq \overline J^{\mathrm {cl}}$ unless $B$ is noetherian.
\item Completion is clearly universal for morphisms to complete $\delta$-pairs.
\end{enumerate}
\end{rmks}

\begin{xmp}
If $B$ is a bounded $\delta$-ring, then $(\widehat B, 0)$ and $(\widehat B, \phi^{-1}((p)_{q}\widehat B))$ are complete bounded $q$-PD-pairs.
This is also the case for $(\widehat B, (q-1)\widehat B)$ if we assume for example that $B/(q-1)$ is $p$-torsion free as in definition 16.2 of \cite{BhattScholze22} (in which case the ideal is closed).
\end{xmp}

We may now make the following definition:

\begin{dfn}
If $(B, J)$ is a $\delta$-pair, then (if it exists) its \emph{complete $q$-PD envelope} $\left(B^{\widehat{[\ ]}_{q}}, J^{\overline{[\ ]}^{\mathrm {cl}}_{q}}\right)$ is a complete $q$-PD-pair which is universal for morphisms to complete $q$-PD-pairs.
\end{dfn}

\begin{prop} \label{bound}
Let $(B, J)$ be a $\delta$-pair.
If the (non complete) $q$-PD-envelope exists \emph{and} is bounded, then its completion is the complete $q$-PD-envelope of $(B,J)$.
\end{prop}

\begin{proof}
It follows from corollary \ref{compbound} that $\widehat{B^{[\ ]_{q}}}$ is bounded and in particular $(p)_{q}$-torsion free.
Moreover, $J\widehat{B^{[\ ]_{q}}}$ is a $q$-PD-ideal. 
Finally, we showed in corollary \ref{closcond} that the $q$-PD-condition is closed thus we see that $\left(\widehat{B^{[\ ]_{q}}}, \overline {J\widehat{B^{[\ ]_{q}}}}^{\mathrm {cl}}\right)$ is a complete $q$-PD-pair.
The universal property is then automatic.
\end{proof}

\begin{rmk}
This proposition will apply when $R$ itself is bounded and $B^{[\ ]_q}$ is flat since then, thanks to lemma \ref{bounded}, $B^{[\ ]_q}$ will be bounded.
\end{rmk}

There exists a complete analog to lemma \ref{paircons} but we will actually only need the following consequence:

\begin{lem} \label{tensor}
If, for $i = 1,2$, $(B, J) \to (B_{i}, J_{i})$ is a morphism of $\delta$-pairs such that each $(B_{i}, J_{i})$ admits a complete $q$-PD-envelope and
\[
B_{1}^{\widehat{[\ ]}_{q}} \widehat \otimes_B B_{2}^{\widehat{[\ ]}_{q}}
\]
is bounded, then this is the complete $q$-PD-envelope of $(B_{1}, J_{1}) \otimes_{(B,J)} (B_{2}, J_{2})$ when it is endowed with the closure of the image of
\[
J_{1}^{\overline{[\ ]}^{\mathrm {cl}}_{q}} \otimes_B B_{2}^{\widehat{[\ ]}_{q}} + B_{1}^{\widehat{[\ ]}_{q}} \otimes_B J_{2}^{\overline{[\ ]}^{\mathrm {cl}}_{q}}.
\]
\end{lem}

\begin{proof}
Corollary \ref{closcond} implies that this is a complete $q$-PD-pair and the universal property then follows automatically from the universal properties of tensor product and completion.
\end{proof}

\begin{xmps}
\begin{enumerate}
\item If $B$ is bounded and $J$ is a $q$-PD-ideal, then $B^{\widehat{[\ ]}_{q}} = \widehat B$.
\item If $J= B$, then $B^{\widehat{[\ ]}_{q}} = 0$. This fundamental property does not rely on proposition \ref{bound}.
It follows from example \ref{J=B} after definition \ref{q-PD-envelope}, which also shows that it is \emph{not} true before completion.
\item If $A$ is a bounded complete unramified $\delta$-$R$-algebra and if we look at the $\delta$-pair $(P, I)$, where $P := A \otimes_{R} A$ and $I$ is the kernel of multiplication $P \to A$, then $P^{\widehat{[\ ]}_{q}} = A$.
This reduces to the previous cases since $(P,I) \simeq (N, N) \times (A, 0)$.
\item If $A$ is a complete bounded $\delta$-$R$-algebra with rank one element $x$, then it follows from theorem \ref{algun} and proposition \ref{bound} that $\left(\widehat {A\langle\xi\rangle_{q}}, \overline {I^{[1]}}^{\mathrm {cl}}\right)$ is the complete $q$-PD-envelope of $(A[\xi], \xi)$.
\end{enumerate}
\end{xmps}

\section{Hyper $q$-stratifications}

We let $R$ be a $\mathbb Z[q]_{(p,q-1)}$-algebra.
Any $R$-module will be implicitly endowed with its $(p, q-1)$-adic topology and completion always means $(p, q-1)$-adic completion.

We let $A$ be a \emph{complete} $R$-algebra with a fixed \emph{topologically \'etale} (that is, formally \'etale and topologically finitely presented) \emph{coordinate} $x$.
We let $P := A \otimes_{R} A$ and we denote by $I$ the kernel of multiplication $e : P \to A$.
We denote by $p_{1}, p_{2} : A \to P$ the canonical maps and we let $\xi := p_{2}(x) - p_{1}(x) \in I$.
Unless otherwise specified, we will use the ``left'' action $f \mapsto p_{1}(f)$ of $A$ on $P$ in order to turn $P$ into an $A$-algebra.

Since $q-1$ is topologically nilpotent on $A$, which is complete and has $x$ as a topologically \'etale coordinate, there exists a unique endomorphism $\sigma$ of the $R$-algebra $A$ such that $\sigma(x) = qx$ and $\sigma \equiv \mathrm {Id}_{A} \mod q-1$.
In particular, $A$ is canonically a \emph{twisted algebra} in the sense of \cite{LeStumQuiros18b}.
We extend $\sigma$ to $P$ in an asymmetric way by setting $\sigma_{P} := \sigma \otimes \mathrm{Id}_{A}$ (we will simply write $\sigma$ when no confusion can arise).
For $n \in \mathbb N$, we let
\[
I^{(n+1)} := I\sigma(I) \cdots \sigma^{n}(I)
\]
be the $(n+1)$th \emph{twisted power} of the ideal $I$ and denote by $\overline {I^{(n+1)}}^{\mathrm {cl}}$ the closure of its image in $\widehat P$.
It is not clear if $\widehat {I^{(n+1)}} \simeq \overline {I^{(n+1)}}^{\mathrm {cl}}$ in general, but this is the case for $n=0$ because the sequence $0 \to I \to P \to A \to 0$ is split exact as a sequence of $A$-modules (and $A$ is complete).

One can show that $x$ is a \emph{topological} \emph{$q$-coordinate} on $A$ in the following sense (we denote by the same letter $\xi$ both the indeterminate and the element of $P$):

\begin{lem} \label{qcor}
The canonical maps are bijective for all all $n\in \mathbb N$:
\[
A[\xi]/(\xi^{(n+1)}) \simeq \widehat P/\overline {I^{(n+1)}}^{\mathrm {cl}}.
\]
\end{lem}

\begin{proof}
We fix some $n \in \mathbb N$.
We have $\xi^{n+1} \equiv \xi^{(n+1)} \mod q-1$ and $\xi$ is therefore topologically nilpotent in $A[\xi]/(\xi^{(n+1)})$.
Since $\xi$ is a topological coordinate for $P$ and $A[\xi]/(\xi^{(n+1)})$ is complete, there exists a unique morphism $U$ below making the diagram commute:
\[
\xymatrix{P \ar@{->>}[r]^e \ar@{-->}[rd]^U & A \\ A[\xi] \ar[u] \ar@{->>}[r] & A[\xi]/(\xi^{(n+1)}). \ar@{->>}[u]}
\]
This morphism $U$ is necessarily surjective and extends by continuity to a surjective morphism
\[
\widehat U : \widehat P \to A[\xi]/(\xi^{(n+1)}).
\]
By construction, we have $U(I) \subset (\xi)$ and therefore $U(I^{(n+1)}) = 0$ so that $\widehat U\left(\overline {I^{(n+1)}}^{\mathrm {cl}}\right)= 0$.
Thus, $\widehat U$ induces a surjective map
\[
u : \widehat P/\overline {I^{(n+1)}}^{\mathrm {cl}} \to A[\xi]/(\xi^{(n+1)}).
\]
Now, if we compose $U$ with the canonical map
\[
v : A[\xi]/(\xi^{(n+1)}) \to \widehat P/\overline {I^{(n+1)}}^{\mathrm {cl}}
\]
and the projection onto $A$, we recover $e : P \to A$.
Using again the fact that $\xi$ is a topological coordinate for $P$ over $A$, since $\widehat P/\overline {I^{(n+1)}}^{\mathrm {cl}}$ is complete and $\widehat I$ consists of topologically nilpotent elements modulo $\overline {I^{(n+1)}}^{\mathrm {cl}}$, we see that $v \circ U$ is necessarily the projection onto $\widehat P/\overline {I^{(n+1)}}^{\mathrm {cl}}$.
It follows that $u$ is an inverse for $v$.
\end{proof}

Let us consider now the canonical map $A[\xi] \to \widehat {A\langle\xi\rangle_{q}}$ from section \ref{deltwist}.
Since the image $(k)_{q}!\xi^{[k]}$ of $\xi^{(k)}$ goes to zero when $k$ goes to infinity, we see that this map factors uniquely through the \emph{twisted power series ring}
\[
A[[\xi]]_{q} := \varprojlim A[\xi]/(\xi^{(n+1)}).
\]

\begin{dfn} \label{qTay}
The \emph{$q$-Taylor map (of level zero)} is the composite
\[
\xymatrix@R=0cm{
\theta : A \ar[r]^-{p_{2}} & P \ar[r] & \varprojlim \widehat P/\overline {I^{(n+1)}}^{\mathrm {cl}} \ar[r]^-\simeq & A[[\xi]]_{q} \ar[r] & \widehat {A\langle\xi\rangle}_{q}.
}
\]
\end{dfn}

We may still denote by the letter $\theta$ the map so defined, but with the arrow stopping at $\varprojlim \widehat P/\overline {I^{(n+1)}}^{\mathrm {cl}}$, $ A[[\xi]]_{q} $ or $\widehat {A\langle\xi\rangle}_{q}$, if no confusion could result.
Notice that we always have $\theta(x) = x + \xi$.

We will need to be able to move between the canonical (left) structure of $\widehat {A\langle\xi\rangle}_{q}$ as an $A$-module and its (right) structure through the map $\theta$ and we first prove the following technical result:

\begin{lem} \label{bigdeal}
If we denote by
\[
\tau : A[\xi] \to A[[\xi]]_{q}, \quad \left\{\begin{array}{l}f \mapsto \theta(f) \ \mathrm{for}\ f \in A\\ \xi \mapsto -\xi. \end{array}\right.
\]
the \emph{flip map} (which is a morphism of $R$-algebras), then we have for all $n > 0$,
\[
\tau\left(\xi^{(n)}\right) = \sum_{k =1}^n (-1)^k q^{\frac{k (k -1)}2} (1-q)^{n-k}(n-k)_{q}!{n-1 \choose k-1}_{q} {n \choose k }_{q} x^{n-k }\xi^{(k )}.
\]
\end{lem}

\begin{proof}
Recall first that we have
\[
\xi^{(n)} = \prod_{k =0}^{n-1}(\xi + (1-q^k )x).
\]
Since $\tau(x) = x + \xi$ and $\tau(\xi) = -\xi$, we will have for all $k =1, \ldots, n-1$,
\[
\tau(\xi + (1-q^k )x) = -\xi + (1-q^k )(x+\xi) = x - q^k (x+\xi),
\]
and it follows that
\[
\tau\left(\xi^{(n)}\right) = \prod_{k =0}^{n-1}(x - q^k (x+\xi)).
\]
We may then apply the twisted binomial formula
\begin{equation} \label{bin}
\prod_{k =0}^{n-1}(q^k X + Y) = \sum_{m=0}^{n} q^{\frac{m(m-1)}2}{n \choose m}_{q} X^mY^{n-m}
\end{equation}
from proposition 2.14 of \cite{LeStumQuiros15} with $X=-(x+\xi)$ and $Y=x$ and obtain 
\[
\tau\left(\xi^{(n)}\right) = \sum_{m=0}^{n} (-1)^m q^{\frac{m(m-1)}2}{n \choose m}_{q} (x+\xi)^mx^{n-m}.
\]
Using our formula
\[
(x+\xi)^m = \sum_{k =0}^m {m \choose k }_{q} x^{m-k }\xi^{(k )}
\]
from lemma 7.1 of \cite{GrosLeStumQuiros19}, we obtain
\[
\tau\left(\xi^{(n)}\right) = \sum_{m =0}^{n} (-1)^m q^{\frac{m(m-1)}2}{n \choose m}_{q} \left( \sum_{k =0}^m {m \choose k }_{q} x^{m-k }\xi^{(k )}\right)x^{n-m}.
\]
Making the change of indices $m=k +l$, we have
\[
\frac{m(m-1)}2 = \frac{k (k -1)}2 + \frac{l(l-1)}2 + k l
\]
and
\[
{n \choose m}_{q} {m \choose k }_{q} = {n-k \choose l}_{q} {n \choose k }_{q}.
\]
It follows that
\[
\tau\left(\xi^{(n)}\right) = \sum_{k =0}^n (-1)^k q^{\frac{k (k -1)}2} {n \choose k }_{q} \left(\sum_{l=0}^{n-k } (-1)^l q^{ \frac{l(l-1)}2} {n-k \choose l}_{q} q^{k l}\right) x^{n-k }\xi^{(k )}.
\]
We may call again our twisted binomial formula \eqref{bin} in the case $X=q^k$ and $Y = -1$ (and also $n$ replaced with $n-k $) which gives
\begin{align*}
\prod_{i=0}^{n-k -1}(1-q^{k+i}) &= (-1)^{n-k }\prod_{i=0}^{n-k-1}(q^iq^k-1)
\\&= (-1)^{n-k } \sum_{l=0}^{n-k } q^{\frac{l(l-1)}2}{n-k \choose l}_{q} (q^k )^l(-1)^{n-k -l}
\\& = \sum_{l=0}^{n-k }(-1)^{l} q^{\frac{l(l-1)}2}{n-k \choose l}_{q} q^{kl}.
\end{align*}
On the other hand, we also have
\[
\prod_{i=0}^{n-k -1}(1-q^{k+i}) = \prod_{i=0}^{n-k -1}(1-q)(k+i)_{q} = (1-q)^{n-k}(n-k)_{q}!{n-1 \choose k-1}_{q},
\]
so finally (since this is $0$ when $k=0$)
\[
\tau\left(\xi^{(n)}\right) = \sum_{k =1}^n (-1)^k q^{\frac{k (k -1)}2} (1-q)^{n-k}(n-k)_{q}!{n-1 \choose k-1}_{q} {n \choose k }_{q} x^{n-k }\xi^{(k )}. \qedhere
\]
\end{proof}


\begin{prop} \label{flippf}
The $R$-linear \emph{flip} map $\tau : \widehat {A \langle \xi \rangle}_{q} \to \widehat {A \langle \xi \rangle}_{q}$ defined by
\[
f\xi^{[n]} \mapsto \theta(f)\sum_{k =1}^n (-1)^k q^{\frac{k (k -1)}2} (1-q)^{n-k}{n-1 \choose k-1}_{q} x^{n-k }\xi^{[k]}
\]
is an $R$-algebra automorphism of $\widehat {A \langle \xi \rangle}_{q}$ such that $\tau \circ \tau = \mathrm{Id}_{\widehat {A \langle \xi \rangle}_{q}}$.
\end{prop}

\begin{proof}
In order to show that this is a morphism of rings, it is sufficient to check that
\[
\forall m,n \in \mathbb N, \quad \tau\left(\xi^{[n]}\xi^{[m]}\right)= \tau\left(\xi^{[n]}\right)\tau\left(\xi^{[m]}\right).
\]
By functoriality, we may assume that $R= \mathbb Z[q]_{(p,q-1)}$ and $A = \widehat{R[x]}$ so that $A$ has no torsion and it is then sufficient to show the equality
\[
\forall m,n \in \mathbb N, \quad \tau\left(\xi^{(n)}\xi^{(m)}\right)= \tau\left(\xi^{(n)}\right)\tau\left(\xi^{(m)}\right).
\]
We are therefore led to prove the similar equality for the map $\tau$ of lemma \ref{bigdeal}, but this is a morphism of rings by definition.

Let us now show that the composite map $\tau \circ \tau$ is $A$-linear, or, equivalently, that the composite map $\tau \circ \theta : A \to \widehat {A \langle \xi \rangle}_{q}$ is the canonical map.
In order to do so, we may first remark that, the map $\tau$ of lemma \ref{bigdeal} sends the ideal $(\xi^{(n)})$ into the ideal $(\xi^{(n-k)}),(q-1)^k)$ when $n \geq k$.
It follows that $\tau$ induces, for all $k \in \mathbb N$, a morphism
\[
A[[\xi]]_{q} = \varprojlim_{n \geq k} A[\xi]/(\xi^{(n)}) \to \varprojlim_{n \geq k} A[\xi]/(\xi^{(n-k)},(q-1)^k) = A[[\xi]]_{q}/(q-1)^k
\]
and, taking the limit on $k$, a ring endomorphism of $A[[\xi]]_{q}$ that we will still denote by $\tau$.
This allows us to decompose $\tau \circ \theta$ as the composite
\[
A \overset \theta \to A[[\xi]]_{q} \overset \tau \to A[[\xi]]_{q} \to \widehat {A \langle \xi \rangle}_{q}.
\]
Since $\xi$ is topologically nilpotent on $A[[\xi]]_{q}$ and $x$ is a topologically \'etale coordinate on $A$, our claim then follows from the fact that
\[
\tau(\theta(x)) = \tau(x + \xi) = (x + \xi) - \xi = x.
\]
It order to prove that $\tau \circ \tau$ is the identity, it is now sufficient to show the equality $(\tau \circ \tau)\left(\xi^{[n]}\right) = \xi^{[n]}$.
As above, we may assume that $A$ has no torsion and we are reduced to showing that $(\tau \circ \tau)(\xi^{(n)}) = \xi^{(n)}$.
This follows again from lemma \ref{bigdeal} because $(\tau \circ \tau)(\xi) = \tau(-\xi) = \xi$ and $\tau \circ \tau$ is $A$-linear.
\end{proof}

\begin{rmks}
\begin{enumerate}
\item \label{rmk1} The flip map $\tau$ exchanges the canonical map with $\theta$ in the sense that $\tau \circ \theta = \mathrm{can}$ and $\tau \circ \mathrm{can} = \theta$. 
\item
The flip map also extends to an automorphism of
\begin{equation} \label{Axiq}
A \langle\langle \xi \rangle\rangle_{q} := \varprojlim A \langle \xi \rangle_{q}/I^{[n+1]} = \varprojlim \widehat{A \langle \xi \rangle_{q}}/\widehat{I^{[n+1]}}
\end{equation}
because $\tau$ sends $\widehat{I^{[n+k]}}$ into $\widehat {I^{[n]}} + (q-1)^k\widehat {A \langle \xi \rangle}_{q}$.
\item
Actually, the same holds with $A[[\xi]]_{q}$ and there exists a commutative diagram
\[
\xymatrix{
f\otimes g \ar@{|->}[d] &P \ar[r] \ar[d]^\simeq & A[[\xi]]_{q} \ar[r] \ar[d]^\simeq & \widehat {A\langle\xi\rangle}_{q} \ar@{^{(}->}[r] \ar[d]^\simeq & A \langle\langle \xi \rangle\rangle_{q} \ar[d]^\simeq
\\
g \otimes f &P \ar[r] & A[[\xi]]_{q} \ar[r] & \widehat {A\langle\xi\rangle}_{q} \ar@{^{(}->}[r] & A \langle\langle \xi \rangle\rangle_{q}
}
\]
of \emph{flip} maps, all of which we may denote by $\tau$.
Note that the left hand square commutes thanks to remark \ref{rmk1}.
\end{enumerate}
\end{rmks}

We will also let $e : \widehat{A\langle\xi\rangle_{q}} \to A$ be the augmentation map and
\[
p_{1}, p_{2} : \widehat{A\langle\xi\rangle_{q}} \to \widehat{A\langle\xi\rangle_{q}} \widehat\otimes'_{A} \widehat{A\langle\xi\rangle_{q}}
\]
be the obvious maps (we use the notation $\otimes'$ to indicate that we use the $q$-Taylor map $\theta$ for the $A$-structure on the left hand side).

\begin{prop} \label{diagprop}
The \emph{diagonal} map
\[
\xymatrix@R=0cm{
\widehat{A\langle\xi\rangle_{q}} \ar[r]^-\Delta & \widehat{A\langle\xi\rangle_{q}} \widehat \otimes'_{A} \widehat{A\langle\xi\rangle_{q}} \\
\xi^{[n]} \ar@{|->}[r] & \sum_{i+j=n} \xi^{[i]} \otimes \xi^{[j]}
}
\]
is a morphism of $A$-algebras (for the left $A$-module structures).
\end{prop}

\begin{proof}
We have to prove that
\[
\forall m, n \in \mathbb N, \quad \Delta\left(\xi^{[n]}\xi^{[m]}\right) = \Delta \left(\xi^{[n]}\right)\Delta\left(\xi^{[m]}\right).
\]
By functoriality, we may clearly assume that $R= \mathbb Z[q]_{(p,q-1)}$ and $A = \widehat{R[x]}$.
Since now $A$ has no torsion, it is sufficient to prove that
\[
\forall m, n \in \mathbb N, \quad \Delta\left(\xi^{(n)}\xi^{(m)}\right) = \Delta(\xi^{(n)})\Delta(\xi^{(m)}).
\]
This follows from theorem 3.5 of \cite{LeStumQuiros18}: the diagonal map
\[
\xymatrix@R=0cm{
R[x,\xi] \ar[r]^-\Delta & R[x,\xi] \otimes'_{A} R[x,\xi] 
\\
\xi^{(n)} \ar@{|->}[r] &\sum_{i+j=n} {n \choose i}_{q} \xi^{(i)} \otimes \xi^{(j)}
}
\]
(where $\otimes'_{A}$ is built using $x \mapsto x + \xi$ on the left) is a ring homomorphism.
\end{proof}

\begin{rmks}
\begin{enumerate}
\item
The same formula also defines a diagonal map
\[
\Delta : A \langle\langle \xi \rangle\rangle_{q} \to A \langle\langle \xi \rangle\rangle_{q} \widehat \otimes'_{A} A \langle\langle \xi \rangle\rangle_{q}
\]
with $A \langle\langle \xi \rangle\rangle_{q}$ as in \eqref{Axiq}.
\item
Actually, this also holds with $A[[\xi]]_{q}$ and the diagram
\[
\xymatrix{
P \ar[r] \ar[d]^\Delta & A[[\xi]]_{q} \ar[r] \ar[d]^\Delta & \widehat {A\langle\xi\rangle}_{q} \ar@{^{(}->}[r] \ar[d]^\Delta & A \langle\langle \xi \rangle\rangle_{q} \ar[d]^\Delta
\\
P \otimes'_{A} P \ar[r] & A[[\xi]]_{q} \widehat \otimes'_{A} A[[\xi]]_{q} \ar[r] & \widehat {A\langle\xi\rangle}_{q} \widehat\otimes'_{A} \widehat {A\langle\xi\rangle}_{q}\ar@{^{(}->}[r] & A \langle\langle \xi \rangle\rangle_{q} \widehat \otimes'_{A} A \langle\langle \xi \rangle\rangle_{q},
}
\]
where the left vertical map is given by $f \otimes g \mapsto f \otimes 1 \otimes g$, is commutative.
For the left hand square, this follows from the formula given in theorem 3.5 of \cite{LeStumQuiros18} in the case $A = \widehat {R[x]}$.
In general, it is sufficient, since $x$ is a topologically \'etale coordinate, to check the commutativity modulo $(p,q-1)$.
But then $\xi$ becomes nilpotent and we can work modulo $\xi$ (or more precisely modulo $(1 \otimes \xi, \xi \otimes 1)$) in which case the assertion is trivial.
\end{enumerate}
\end{rmks}

We now show that $\left(A, \widehat {A\langle\xi\rangle}_{q}, \theta, \mathrm{can}, e, \Delta\right)$ and $\tau$ satisfy properties similar (because we use completions) to those of a \emph{formal groupoid} in the sense of definition 1.1.3 in chapter II of \cite{Berthelot74} (we follow the same order for the statements):

\begin{prop} \label{groupoid}
The following diagrams are commutative:
\begin{enumerate}
\item
\[
\xymatrix{
A \ar[r] \ar@{=}[rd] & \widehat {A\langle\xi\rangle}_{q} \ar[d]^e
\\
& A
}
\quad \mathrm{and} \quad
\xymatrix{
A \ar[r]^-\theta \ar@{=}[rd] & \widehat {A\langle\xi\rangle}_{q} \ar[d]^e
\\
& A
}
\]
\item
\[
\xymatrix{
A \ar[r] \ar[d] & \widehat {A\langle\xi\rangle}_{q} \ar[d]^\Delta
\\
\widehat {A\langle\xi\rangle}_{q} \ar[r]^-{p_{1}} & \widehat {A\langle\xi\rangle}_{q} \widehat \otimes'_{A} \widehat {A\langle\xi\rangle}_{q}
}
\quad \mathrm{and} \quad
\xymatrix{
A \ar[r]^\theta \ar[d]^\theta & \widehat {A\langle\xi\rangle}_{q} \ar[d]^\Delta
\\
\widehat {A\langle\xi\rangle}_{q} \ar[r]^-{p_{2}} & \widehat {A\langle\xi\rangle}_{q} \widehat \otimes'_{A} \widehat {A\langle\xi\rangle}_{q}
}
\]
\item
\[
\xymatrix{
\widehat {A\langle\xi\rangle}_{q} \ar[r]^-\Delta \ar@{=}[rd] & \widehat {A\langle\xi\rangle}_{q} \widehat \otimes'_{A} \widehat {A\langle\xi\rangle}_{q} \ar[d]^{e \otimes \mathrm{Id}}
\\ & \widehat {A\langle\xi\rangle}_{q}
}
\quad \mathrm{and} \quad
\xymatrix{
\widehat {A\langle\xi\rangle}_{q} \ar[r]^-\Delta \ar@{=}[rd] & \widehat {A\langle\xi\rangle}_{q} \widehat \otimes'_{A} \widehat {A\langle\xi\rangle}_{q} \ar[d]^{\mathrm{Id} \otimes e}
\\ & \widehat {A\langle\xi\rangle}_{q}
}
\]
\item \label{assoc}
\[
\xymatrix{
\widehat {A\langle\xi\rangle}_{q} \ar[r]^-\Delta \ar[d]^\Delta & \widehat {A\langle\xi\rangle}_{q} \widehat \otimes'_{A} \widehat {A\langle\xi\rangle}_{q} \ar[d]^{\Delta \otimes \mathrm{Id}}
\\ \widehat {A\langle\xi\rangle}_{q} \widehat \otimes'_{A} \widehat {A\langle\xi\rangle}_{q} \ar[r]^-{\mathrm{id} \otimes \Delta} & \widehat {A\langle\xi\rangle}_{q} \widehat \otimes'_{A}\widehat {A\langle\xi\rangle}_{q} \widehat \otimes'_{A} \widehat {A\langle\xi\rangle}_{q}
}
\]
\item
\[
\xymatrix{
A \ar[d] \ar[rd]^\theta &
\\ \widehat {A\langle\xi\rangle}_{q} \ar[r]^\tau& \widehat {A\langle\xi\rangle}_{q}
}
\quad \mathrm{and} \quad
\xymatrix{
A \ar[d]^\theta \ar[rd] &
\\ \widehat {A\langle\xi\rangle}_{q} \ar[r]^\tau& \widehat {A\langle\xi\rangle}_{q}
}
\]
\item
\[
\xymatrix{
\widehat {A\langle\xi\rangle}_{q} \ar[r]^\tau \ar[rd]^e & \widehat {A\langle\xi\rangle}_{q} \ar[d]^e
\\
& A
}
\]
\item
\[
\xymatrix{
\widehat {A\langle\xi\rangle}_{q} \ar[r]^e \ar[d]^\Delta & A \ar[d]
\\ \widehat {A\langle\xi\rangle}_{q} \widehat \otimes'_{A} \widehat {A\langle\xi\rangle}_{q} \ar[r]^-{\mathrm{Id} \times \tau} & \widehat {A\langle\xi\rangle}_{q}}
\quad \mathrm{and} \quad
\xymatrix{
\widehat {A\langle\xi\rangle}_{q} \ar[r]^e \ar[d]^\Delta & A \ar[d]^\theta
\\ \widehat {A\langle\xi\rangle}_{q} \widehat \otimes'_{A} \widehat {A\langle\xi\rangle}_{q} \ar[r]^-{\tau \times \mathrm{Id}} & \widehat {A\langle\xi\rangle}_{q}.}
\]
\end{enumerate}
\end{prop}

\begin{proof}
Only the last assertion needs a proof, the other ones following either directly from the definitions or from what we already proved.
For example, the commutativity of the right hand square in assertion (2) follows from the second remark before the proposition.
Now, we consider the last assertion and since both proofs are similar, we only do the first diagram.
We have to show that
\[
\forall n > 0, \quad (\mathrm{Id} \times \tau)\left(\Delta \left(\xi^{[n]}\right)\right) = 0.
\]
As usual, we can reduce by functoriality to showing that $(\mathrm{Id} \otimes \tau)(\Delta(\xi^{(n)})) = 0$ and this results from the commutativity of the diagram
\[
\xymatrix{
P \ar[r]^e \ar[d]^\Delta & A \ar[d]
\\ P \otimes'_{A}P\ar[r]^-{\mathrm{Id} \times \tau} & P
}
,\quad
\xymatrix{
f \otimes g \ar@{|->}[r]^e \ar@{|->}[d]^\Delta & fg \ar@{|->}[d]
\\ f \otimes 1 \otimes g\ar@{|->}[r]^-{\mathrm{Id} \times \tau} & fg \otimes 1.
} \qedhere
\]
\end{proof}

There exists a notion of stratification in this setting that reads as follows:

\begin{dfn}\label{hyperqpd}
A \emph{hyper-$q$-stratification (of level zero)} \footnote{We could also say \emph{hyper-$q$-PD-stratification}.}
on an $A$-module $M$ is an $\widehat {A\langle\xi\rangle_{q}}$-linear isomorphism
\[
\epsilon : \widehat {A\langle\xi\rangle}_{q} \otimes'_{A} M \simeq M \otimes_{A} \widehat {A\langle\xi\rangle}_{q}
\]
such that
\begin{equation} \label{cocyc}
e^*(\epsilon) = \mathrm{Id}_{M} \quad \mathrm{and} \quad \Delta^*(\epsilon) = p_{1}^*(\epsilon) \circ p_{2}^*(\epsilon).
\end{equation}
\end{dfn}

There also exists an apparently weaker notion which is defined as follows :

\begin{dfn}
A \emph{hyper-$q$-Taylor structure (of level zero)} on an $A$-module $M$ is an $A$-linear map
\[
\theta : M \to M \otimes_{A} \widehat {A\langle\xi\rangle_{q}},
\]
where we consider the target as an $A$-module via the $q$-Taylor map of $A$, such that
\[
(\mathrm{Id}_{M} \otimes e) \circ \theta = \mathrm{Id}_{M}
\]
and
\[
(\theta \otimes \mathrm{Id}_{\widehat {A\langle\xi\rangle_{q}}}) \circ \theta = (\mathrm{Id}_{M} \otimes \Delta) \circ \theta.
\]
\end{dfn}

\begin{prop} \label{eqTay}
Restriction and scalar extension along the $q$-Taylor map $\theta : A \to \widehat {A\langle\xi\rangle}_{q}$ provide an equivalence between hyper-$q$-stratifications and hyper-$q$-Taylor structures.
\end{prop}

In other words, $A$-modules endowed with a hyper-$q$-stratification (resp. a hyper-$q$-Taylor structure) form a category in an obvious way and both categories are isomorphic.

\begin{proof}
This is analogous to corollary 1.4.4 in chapter II of \cite{Berthelot74}.
It is easy to see that any hyper-$q$-stratification induces a hyper-$q$-Taylor structure and that any hyper-$q$-Taylor structure extends to an $\widehat {A\langle\xi\rangle}_{q}$-linear map
\[
\epsilon : \widehat {A\langle\xi\rangle}_{q} \otimes'_{A} M \to M \otimes_{A} \widehat {A\langle\xi\rangle}_{q}, \quad 1 \otimes s \mapsto \theta(s)
\]
that satisfies the normalization and cocycle conditions.
It only remains to show that $\epsilon$ is bijective and it is then that the flip map enters the picture: if we pull back the cocycle condition along $\mathrm{Id} \times \tau$ (resp. $\tau \times \mathrm{Id}$), we get $\mathrm{Id} = \epsilon \circ \tau^*(\epsilon)$ (resp. $\mathrm{Id} = \tau^*(\epsilon) \circ \epsilon$).
\end{proof}

\section{$q$-calculus}

As before, we let $R$ be a $\mathbb Z[q]_{(p,q-1)}$-algebra and endow all modules with their $(p, q-1)$-adic topology.
Completion is always meant with respect to this topology.
We let $A$ be a complete $R$-algebra with a fixed topologically \'etale coordinate $x$ and we keep the notations from the previous section.
We want to extend here some notions from our previous articles \cite{LeStumQuiros18}, \cite{LeStumQuiros18b} and \cite{GrosLeStumQuiros19} by taking into account the topology (see also \cite{LeStumQuiros20}, where we did the affinoid case).

Recall that there exists a unique endomorphism $\sigma$ of the $R$-algebra $A$ such that $\sigma(x) = qx$ and $\sigma \equiv \mathrm {Id}_{A} \mod q-1$.
Then, the basic notion in $q$-calculus is the following:

\begin{dfn}
\begin{enumerate}
\item
A \emph{$q$-derivation} of $A$ with values in an $A$-module $M$ is an $R$-linear map $D : A \to M$ that satisfies the twisted Leibniz rule
\[
\forall f, g \in A, \quad D(fg) = D(f)g + \sigma(f)D(g).
\]
\item
A \emph{$q$-derivation} on an $A$-module $M$ with respect to some $q$-derivation $D_{A} : A \to A$ is a map $D_{M} : M \to M$ that satisfies the twisted Leibniz rule
\[
\forall f \in A, \forall s \in M, \quad D_{M}(fs) = D_{A}(f)s + \sigma(f)D_{M}(s).
\]
\item
We will denote by $\mathrm T_{A/R,q}$ the module of $q$-derivations on $A$ with values in $A$.
An \emph{action of $\mathrm T_{A/R,q}$ by $q$-derivations} on $M$ is an $A$-linear map $\mathrm T_{A/R,q} \to \mathrm{End}_{R}(M)$ such that the image of $D \in \mathrm T_{A/R,q}$ is a $q$-derivation of $M$ with respect to $D$.
\end{enumerate}
\end{dfn}

The $q$-derivation $D_{A}$ plays a secondary role in the definition of $D_{M}$ and we will not mention it in the future (we may even use the same letter $D$ for both).
Note also that both definitions of a $q$-derivation coincide when $A=M$ (in which case $D_M= D_A$).

In order to study $q$-derivations, it is necessary to introduce the (complete) module of $q$-differential forms.
Let us first recall that we extend in an asymmetric way the endomorphism $\sigma$ to $P := A \otimes_{R} A$ and that we denote by $I$ the kernel of multiplication and by $I^{(n+1)}$ its twisted powers.

\begin{dfn}
The \emph{(complete) module of $q$-differential forms\footnote{There is no natural isomorphism between $\Omega_{A/R,q}$ and the usual (complete) module of derivations $\Omega_{A/R}$.}} of $A/R$ is
\[
\Omega_{A/R,q} := \widehat I/ \overline {I^{(2)}}^{\mathrm {cl}}.
\]
\end{dfn}

The $R$-linear map $p_{2}-p_{1} : A \to P$ takes values inside $I$ and induces therefore an $R$-linear map $\mathrm d_{q} : A \to \Omega_{A/R,q}$.
It has the following universal property:

\begin{prop}
The map $\mathrm d_{q} : A \to \Omega_{A/R,q}$ is a $q$-derivation which is universal among all $q$-derivations with value in \emph{complete} $A$-modules. 
Moreover, $\Omega_{A/R,q}$ is a free $A$-module of rank one with basis $\mathrm d_{q}x$.
\end{prop}

\begin{proof}
We already know from proposition 2.4 of \cite{LeStumQuiros18} that $I/I^{(2)}$ is universal with respect to $q$-derivations into any $A$-module and $\Omega_{A/R,q}$ is by definition the completion of $I/I^{(2)}$.
The second assertion follows from lemma \ref{qcor}.
\end{proof}

As an immediate consequence, we see that $\mathrm T_{A/R,q}$ is a free $A$-module on one generator $\partial_{A,q}$ determined by the condition $\partial_{A,q}(x) = 1$.
In particular, giving an action of $\mathrm T_{A/R,q}$ by $q$-derivations of $M$ amounts to specifying a $q$-derivation $\partial_{M,q}$ on $M$ with respect to $\partial_{A,q}$.
There exists an equivalent notion which is more natural in some sense:

\begin{dfn}
A \emph{$q$-connection} on an $A$-module $M$ is an $R$-linear map
\[
\nabla : M \to M \otimes_{A} \Omega_{A/R,q}
\]
that satisfies the twisted Leibniz rule
\[
\forall f \in A, \forall s \in M, \quad \nabla(fs) = s \otimes \mathrm d_{q}(f) + \sigma(f)\nabla(s).
\]
\end{dfn}

A $q$-connection on $M$ is equivalent to an action of $\mathrm T_{A/R,q}$ by $q$-derivations via the formula
\[
\nabla(s) = \partial_{M,q}(s) \otimes \mathrm d_{q}(x).
\]

We shall soon discuss the general notion of a $q$-differential operator but we can already introduce the following:

\begin{dfn}
The \emph{ring of $q$-differential operators (of level zero)} of $A/R$ is the non-commutative polynomial ring $ D_{A/R,q}$ in one generator $\partial_{q}$ over $A$ with the commutation rule
\[
\partial_{q} \circ f = \sigma(f)\partial_{q} + \partial_{A,q}(f).
\]
\end{dfn}

Clearly, a $q$-connection is equivalent to a structure of a $ D_{A/R,q}$-module through the formula $\partial_{q}s = \partial_{M,q}(s)$.

The twisted polynomial ring will now enter the picture:

\begin{dfn}
For each $n \in \mathbb N$, the \emph{$q$-Taylor map of (level zero and) order $n$} is the composite
 \[
 \theta_{n} : A \overset \theta \to \widehat{A\langle\xi\rangle}_{q} \twoheadrightarrow \widehat{A\langle\xi\rangle_{q}}/\widehat{I^{[n+1]}} \simeq A\langle\xi\rangle_{q}/I^{[n+1]}.
 \]
 \end{dfn}
 
This is a morphism of $R$-algebras that endows $A\langle\xi\rangle_{q}/I^{[n+1]}$ with a new structure of $A$-module that we may informally call the right structure (by symmetry with the left structure corresponding to the canonical action of $A$).
As we already did with $\widehat{A\langle\xi\rangle}_{q}$, we will use the notation $\otimes'$ to indicate that we consider $A\langle\xi\rangle_{q}/I^{[n+1]}$ as an $A$-module via the $q$-Taylor map $\theta_n$ and not with respect to the canonical map.

We will prove general formulas below but we can already notice that the $q$-Taylor map is closely related to $q$-derivations because, by definition,
\[
\forall f \in A, \quad \theta_1(f) - f = \partial_{A,q}(f)\xi.
\]
In order to make explicit computations later, it will be necessary to have a better grasp on the right $A$-module structure:

\begin{lem} \label{sigman}
The action of $A$ via the $q$-Taylor map on the ideal $I^{[n]}/I^{[n+1]}$ is identical to the canonical action of $A$ twisted by $\sigma^n$.
\end{lem}

\begin{proof}
Since $\theta$ is a morphism of $R$-algebras and $ I^{[n]}/ I^{[n+1]}$ is the free $A$-module of rank one generated by (the image of) $\xi^{[n]}$, there exists a unique morphism of $R$-algebras $\sigma_n : A \to A$ such that 
\[
\forall f \in A, \quad \theta(f)\xi^{[n]} \equiv \sigma_n(f) \xi^{[n]} \mod I^{[n+1]}.
\]
In the case $q=1$, we have $\sigma_n = \mathrm{Id} = \sigma^n$.
Moreover, since $\sigma(x+\xi) = x + \xi$ and $\xi^{(n+1)}= \sigma^n(\xi)\xi^{(n)}$, we have
\begin{align*}
\theta(x)\xi^{[n]} &= (x + \xi)\xi^{[n]}
\\ &= \sigma^n(x + \xi)\xi^{[n]}
\\ &= (\sigma^n(x) + \sigma^n(\xi))\xi^{[n]} 
\\ &= \sigma^n(x) \xi^{[n]}+(n+1)_{q}\xi^{[n+1]} 
\end{align*}
so that $\sigma_n(x) = \sigma^n(x)$.
Since $x$ is a topological coordinate on $A$, which is complete, we can conclude that $\sigma_n = \sigma^n$.
\end{proof}

As a consequence, the map $\theta_n$ is finite free.
More precisely:

\begin{lem} \label{goodbasis}
\begin{enumerate}
\item The ideal $I^{[n]}/I^{[n+1]}$ is a free $A$-module of rank one on $\xi^{[n]}$ for both the left and the right $A$-module structures.
\item The ring $A\langle\xi\rangle_{q}/I^{[n+1]}$ is a free $A$-module on $\left\{\xi^{[k]}\right\}_{k=0}^n$ for both the left and the right $A$-module structures.
\item The ring $A\langle\xi\rangle_{q}/I^{[n+1]} \otimes'_A A\langle\xi\rangle_{q}/I^{[m+1]}$ is a free $A$-module on $\left\{\xi^{[k]} \otimes \xi^{[l]} \right \}_{k,l=0}^{n,m}$ for the left, middle and right $A$-module structures.
\end{enumerate}
\end{lem}

\begin{proof}
The first assertion follows from lemma \ref{sigman} since $\sigma$ is an automorphism of $A$ and the other ones are then obtained by induction on $n$ (and $m$).
\end{proof}

As we announced, there exists a more general (and more formal) approach to $q$-differential operators that we explain now (see section 5 of \cite{GrosLeStumQuiros19} for the analogous situation when there is no topology).

\begin{dfn}
Let $M$ and $N$ be two $A$-modules.
A \emph{$q$-differential operator} of order at most $n$ (and level zero) from $M$ to $N$ is an $A$-linear map
\begin{equation} \label{diffop}
u : A\langle \xi \rangle_{q}/I^{[n+1]} \otimes'_{A} M \to N.
\end{equation}
\end{dfn}

We insist on the fact that we have $\varphi \otimes fs = \theta_n(f)\varphi \otimes s$ on the left hand side and linearity means that $u(f\varphi \otimes s) = fu(\varphi \otimes s)$.

We will write
\[
\mathrm{Diff}_{q,n}(M,N) := \mathrm{Hom}_{A}\left(A \langle \xi \rangle_{q}/I^{[n+1]} \otimes'_{A} M, N\right)
\]
and
\[
\mathrm{Diff}_{q}(M, N) := \varinjlim_{n \in \mathbb N}\mathrm{Diff}_{q,n}(M,N)
\]
($\mathrm{Diff}_{q,n}(M)$ and $\mathrm{Diff}_{q}(M)$ when $N=M$).
Note that $\mathrm{Diff}_{q,n}(M,N)$ (resp. $\mathrm{Diff}_{q}(M,N)$) is naturally an $A\langle \xi \rangle_{q}/I^{[n+1]}$-module (resp. an $A\langle \xi \rangle_{q}$-module).
We may now consider for each $m,n \in \mathbb N$ the \emph{diagonal} map of finite order (a morphism of $A$-algebras)
\[
\xymatrix@R=0cm{
A\langle \xi \rangle_{q}/I^{[n+m+1]} \ar[rr]^-{\Delta_{n,m}} && A\langle \xi \rangle_{q}/I^{[n+1]} \otimes'_{A} A\langle \xi \rangle_{q}/I^{[m+1]} \\
\xi^{[k]} \ar@{|->}[rr] && \sum_{i+j=k} \xi^{[i]} \otimes \xi^{[j]},
}
\]
which is induced by the diagonal map $\Delta$ of proposition \ref{diagprop}.
We may compose $q$-differential operators in the usual way.
More precisely, if we are given $u$ as in \eqref{diffop} and
\[
v : A\langle \xi \rangle_{q}/I^{[m+1]} \otimes'_{A} L \to M,
\]
then we define $u \widetilde \circ v$ as the composite
\[
\xymatrix{
A\langle \xi \rangle_{q}/I^{[n+m+1]} \otimes'_{A} L \ar[r]^-{u \widetilde \circ v} \ar[d]^{\Delta_{n,m} \otimes' \mathrm{Id}} & N \\ A\langle \xi \rangle_{q}/I^{[n+1]} \otimes'_{A} A\langle \xi \rangle_{q}/I^{[m+1]} \otimes'_{A} L \ar[r]^-{\mathrm{Id} \otimes' v} & A\langle \xi \rangle_{q}/I^{[n+1]} \otimes'_{A} M \ar[u]^{u}.
}
\]
Composition of $q$-differential operators is somehow tricky because we have for example $\mathrm{Diff}_{q,0}(M,N) = \mathrm{Hom}_{A}\left(M, N\right)$, and in particular $\mathrm{Diff}_{q,0}(A) \simeq A$, so that any $f \in A$ may be considered as a $q$-differential operator of order $0$.
But, if $u \in \mathrm{Diff}_{q,n}(A)$, then $u \widetilde \circ f = u \circ \theta_n(f)$.

Associativity of composition of $q$-differential operators follows from the commutativity of diagram \eqref{assoc} in proposition \ref{groupoid} and, if $M$ is any $A$-module, we obtain a structure of $R$-algebra on $\mathrm{Diff}_{q}(M)$.
It should also be noticed that composition of $q$-differential operators is compatible with composition of $R$-linear maps under the restriction map
\[
\xymatrix@R=0cm{ \mathrm{Diff}_{q,n}(M,N) \ar[r] & \mathrm{Hom}_{R}(M,N) \\
(A\langle \xi \rangle_{q}/I^{[n+1]} \otimes'_{A} M \overset u \to N) \ar[r] & (M \hookrightarrow A\langle \xi \rangle_{q}/I^{[n+1]} \otimes'_{A} M \overset u \to N)
}
\]
(which is not injective in general).
In particular, there exits a canonical $A$-linear morphism of $R$-algebras
\[
\mathrm{Diff}_{q}(M) \to \mathrm{End}_{R}(M)
\]
when $M$ is an $A$-module.

In the case $M=A$, we recover our previous ring of $q$-differential operators:

\begin{lem} \label{Difis}
There exists an $A$-linear isomorphism of $R$-algebras
\[
D_{A/R,q} \simeq \mathrm{Diff}_{q}(A) , \quad \partial_{q}^k \mapsto \left( \widetilde{\partial_{q}^k} : \xi^{[n]}\ \mapsto \left\{\begin{array}l 1\ \mathrm{if}\ n=k \\ 0 \ \mathrm{otherwise} \end{array}\right. \right) \in \mathrm{Diff}_{q,k}(A).
\]
\end{lem}

\begin{proof}
This map is clearly bijective.
Thus, we have to check that
\[
\forall f \in A, \quad \widetilde{\partial_{q}} \widetilde \circ f = \sigma(f) \widetilde{\partial_{q}} + \partial_{A,q}(f) \quad \mathrm{and} \quad \forall k \geq 1, \quad \widetilde{\partial_{q}^{k+1}} = \widetilde{\partial_{q}} \widetilde \circ \widetilde{\partial_{q}^k}.
\]
First of all, we have
\[
( \widetilde{\partial_{q}} \widetilde \circ f)(1) = \widetilde{\partial_{q}} (\theta_1(f)) = \widetilde{\partial_{q}} (f + \partial_{A,q}(f)\xi) = \partial_{A,q}(f),
\]
and
\[
( \widetilde{\partial_{q}} \widetilde \circ f)(\xi) = \widetilde{\partial_{q}} (\theta_1(f)\xi) = \widetilde{\partial_{q}} (\sigma(f)\xi) = \sigma(f).
\]
Since, in general, we have $\left(\sum f_{k} \widetilde{\partial_{q}}^k\right) \cdot \xi^{[n]} = f_{n}$, the first equality holds.
Also,
\begin{align*}
( \widetilde{\partial_{q}} \widetilde \circ \widetilde{\partial_{q}^k})(1) &= ( \widetilde{\partial_{q}} \circ \mathrm{Id} \otimes' \widetilde{\partial_{q}^k} \circ \Delta_{1,k})(1)
\\ &= ( \widetilde{\partial_{q}} \circ \mathrm{Id} \otimes' \widetilde{\partial_{q}^k} )(1 \otimes 1)
\\ &= \widetilde{\partial_{q}} (0)
\\ &= 0,
\end{align*}
 for $1 \leq n \leq k$,
 \begin{align*}
( \widetilde{\partial_{q}} \widetilde \circ \widetilde{\partial_{q}^k} )(\xi^{[n]}) &= ( \widetilde{\partial_{q}} \circ \mathrm{Id} \otimes' \widetilde{\partial_{q}^k} \circ \Delta_{1,k})(\xi^{[n]})
\\ &= ( \widetilde{\partial_{q}} \circ \mathrm{Id} \otimes' \widetilde{\partial_{q}^k} )(1 \otimes \xi^{[n]} + \xi \otimes \xi^{[n-1]})
\\ &= \widetilde{\partial_{q}} \left(\theta(\widetilde{\partial_{q}^k}(\xi^{[n]})) + \theta(\widetilde{\partial_{q}^k} (\xi^{[n-1]})) \xi \right)
\\ &=\left\{\begin{array} l \widetilde{\partial_{q}}(0) = 0 \ \mathrm{if}\ k < n \\ \widetilde{\partial_{q}}(1) = 0 \ \mathrm{if}\ k= n, \end{array}\right.
\end{align*}
and finally
\begin{align*}
( \widetilde{\partial_{q}} \widetilde \circ \widetilde{\partial_{q}^k} )(\xi^{[k+1]}) &= ( \widetilde{\partial_{q}} \circ \mathrm{Id} \otimes' \widetilde{\partial_{q}^k} \circ \Delta_{1,k})(\xi^{[k+1]})
\\ &= ( \widetilde{\partial_{q}} \circ \mathrm{Id} \otimes' \widetilde{\partial_{q}^k} )(\xi \otimes \xi^{[k]})
\\ &= \widetilde{\partial_{q}} (\xi)
\\ &= 1. \qedhere
 \end{align*}
\end{proof}

\begin{rmks}
\begin{enumerate}
\item 
In other words, $\{\partial_q^k\}_{k \in \mathbb N}$ and $\{\xi^{[n]}\}_{n \in \mathbb N}$ are ``dual basis'' for $ D_{A/R,q}$ and $A\langle \xi \rangle_{q}$.
\item 
If we denote the image of $P$ under the isomorphism $D_{A/R,q} \simeq \mathrm{Diff}_{q}(A)$ by $\widetilde P$, then
\[
\forall P \in D_{A/R,q}, \quad P = \sum_{k \geq 0} \widetilde P(\xi^{[k]}) \partial_q^k \quad \mathrm{and} \quad \forall \varphi \in A\langle \xi \rangle_{q}, \quad \varphi = \sum_{n \geq 0} \widetilde \partial_q^n(\varphi) \xi^{[n]}.
\]
\end{enumerate}
\end{rmks}

As a consequence of lemma \ref{Difis}, we see that $D_{A/R,q}$ is naturally an $A\langle \xi \rangle_{q}$-module and we can make the structure explicit:

\begin{lem} \label{actd}
We have
\[
\forall k,n \in \mathbb N, \quad \xi^{[n]} \cdot \partial_{q}^k =\sum_{i = 0}^{n} q^{\frac{i(i-1)}2} {k \choose n}_{q}{n \choose i}_{q} (q-1)^{i} x^i \partial_{q}^{k-n+i}
\]
(for $n \leq k$ and $0$ otherwise).
\end{lem}

\begin{proof}
From formula \eqref{mulrul}, we deduce that, for $m \in \mathbb N$,
\begin{align*}
\left(\xi^{[n]} \cdot \widetilde{\partial_{q}^k}\right) (\xi^{[m]}) & = \widetilde{\partial_{q}^k} \left(\xi^{[n]} \xi^{[m]}\right)
\\ &= \widetilde{\partial_{q}^k}\left(\sum_{0\leq i \leq n,m} q^{\frac{i(i-1)}2}{n + m -i \choose n}_{q}{n \choose i}_{q} (q-1)^{i}x^i\xi^{[n+m-i]}\right)
\\ &=\left\{ \begin{array}{ll}q^{\frac{i(i-1)}2}{n + m -i \choose n}_{q}{n \choose i}_{q} (q-1)^{i} x^i \ &\mathrm{if} \ n+m-i =k, \\ 0 \ &\mathrm{otherwise}. \end{array}\right.
\end{align*}

Note that the condition $0 \leq i \leq m,n$ is only meant to insist on the fact that the $q$-binomial coefficients are zero otherwise and this is why it could disappear from the final formula.
Therefore,
\begin{align*}
\xi^{[n]} \cdot \widetilde{\partial_{q}^k} &= \sum_{i =0}^{n} q^{\frac{i(i-1)}2}{k \choose n}_{q}{n \choose i}_{q} (q-1)^{i} x^i \widetilde{\partial_{q}^{k-n+i}}
\end{align*}
for $n \leq k$, and $0$ otherwise.
\end{proof}

\begin{rmks}
\begin{enumerate}
\item
The formula in the lemma may also be obtained in two steps:
\[
\xi^{[n]} \cdot \partial_{q}^k ={k \choose n}_{q} (\xi^{[n]} \cdot \partial_{q}^n) \circ \partial_{q}^{k-n} \quad \mathrm{with} \quad \xi^{[n]} \cdot \partial_{q}^n =\sum_{i = 0}^{n} q^{\frac{i(i-1)}2}{n \choose i}_{q} (q-1)^{i} x^i \partial_{q}^{i}.
\]
\item Alternatively, if $P \in D_{A/R,q}$ and $\varphi \in A\langle \xi \rangle_{q}$, we have
\[
\varphi \cdot P = \sum_{0 \leq i \leq n \leq k} q^{\frac{i(i-1)}2} {k \choose n}_{q}{n \choose i}_{q} (q-1)^{i} x^i \widetilde P(\xi^{[k]}) \widetilde \partial_{q}^n(\varphi) \partial_{q}^{k-n+i}.
\]
\end{enumerate}
\end{rmks}

We may always consider a $q$-derivation as a $q$-differential operator as follows:

\begin{lem} \label{difop}
A $q$-derivation $D : M \to M$ extends uniquely to a $q$-differential operator
\[
\widetilde D : A\langle \xi \rangle_{q}/I^{[2]} \otimes'_{A }M \to M, \quad \xi \otimes s \mapsto s + (q-1)xD(s)
\]
of order at most one.
\end{lem}

\begin{proof}
First of all, $D$ extends uniquely to an $A$-linear map
\[
\widetilde D : P \otimes'_A M \to M, \quad (f \otimes g) \otimes s \mapsto fD(gs)
\]
and we have
\begin{align*}
\widetilde D(\xi \otimes s) &= \widetilde D ((1 \otimes x) \otimes s - (x \otimes 1) \otimes s) \\&= D(xs) - xD(s) \\&= s + (q-1)xD(s).
\end{align*}
Moreover, for $f,g \in A$ and $s \in M$, we have
\begin{align*}
&\widetilde D((1 \otimes f - f \otimes 1)(1 \otimes g - \sigma(g) \otimes 1) \otimes s)
\\&= \widetilde D((1 \otimes fg \otimes s - \sigma(g) \otimes f \otimes s - f \otimes g \otimes s + f\sigma(g) \otimes 1 \otimes s)
\\&=D(fgs) - \sigma(g)D(fs) - f D(gs) + f\sigma(g) D(s)
\\&= fD(g)s - f D(gs) + f\sigma(g) D(s)
\\&=0.
\end{align*}
It follows that $\widetilde D$ factors uniquely through $P/I^{(2)} \otimes'_A M$ and we obtain the expected differential operator as the composite
\[
A\langle \xi \rangle_{q}/I^{[2]} \otimes'_A M \simeq A[\xi]/\xi^{(2)} \otimes'_A M \to P/I^{(2)} \otimes'_A M \to M. \qedhere
\]
\end{proof}

As an example, we see that $\widetilde {\partial_{A,q}}(\xi) = 1$ so that this notation is compatible with that of lemma \ref{Difis}.

\begin{rmk}
Since we have a split short exact sequence
\[
0 \to \Omega_{A,q} \to A\langle\xi\rangle_{q}/I^{[2]} \to A \to 0,
\]
given an $A$-module $M$, it is equivalent to give an
$R$-linear map $\nabla : M \to M \otimes_A \Omega_{A,q}$ or an $R$-linear map $\theta_1 : M \to M \otimes_A A\langle\xi\rangle_{q}/I^{[2]}$ such that the composite map
\[
M \overset {\theta_1} \to M \otimes_A A\langle\xi\rangle_{q}/I^{[2]} \to M
\]
is the identity.
This equivalence is given by
\[
\forall s \in M, \quad \theta_1(s) = s \otimes 1 + \nabla(s).
\]
Moreover, $\nabla$ is a $q$-connection if and only if $\theta_1$ is semilinear in the sense that
\[
\forall f \in A, s \in M, \quad \theta_1(fs) = \theta_1(f) \theta_1(s).
\]
More precisely, since $\theta_1$ acts as $\sigma$ on $ \Omega_{A,q}$, we have for $f \in A$ and $s \in M$,
\begin{align*}
\theta_1(fs) = \theta_1(f) \theta_1(s) &\Leftrightarrow f s \otimes 1 + \nabla(fs) = \theta_1(f) (s \otimes 1 + \nabla(s))
\\ &\Leftrightarrow \nabla(fs) = s \otimes \theta_1(f) - s \otimes f + \sigma(f) \nabla(s)
\\ &\Leftrightarrow \nabla(fs) = s \otimes \mathrm d_q(f) + \sigma (f) \nabla(s).
\end{align*}

The map $\theta_1$ then extends uniquely to an $A\langle\xi\rangle_{q}/I^{[2]}$-linear map
\[
\epsilon_1 : A\langle\xi\rangle_{q}/I^{[2]} \otimes_A' M \to M \otimes_A A\langle\xi\rangle_{q}/I^{[2]}.
\]
If $\partial_{M,q}$ denotes the $q$-derivation associated to $\nabla$, then we have
\begin{equation} \label{case1}
\widetilde {\partial_{M,q}} : A\langle\xi\rangle_{q}/I^{[2]} \otimes_A' M \overset{\epsilon_1} \longrightarrow M \otimes_A A\langle\xi\rangle_{q}/I^{[2]} \overset{\mathrm{Id} \otimes \widetilde {\partial_{A,q}}} \longrightarrow M.
\end{equation}
\end{rmk}

We will need below the following technical result:

\begin{lem} \label{actd2}
If $M$ is a $D_{A/R,q}$-module and $s \in M$, then
\[
\forall k,n \in \mathbb N, \quad \widetilde \partial_{M,q}^{k}(\xi^{[n]} \otimes s) = \sum_{i = 0}^{n} q^{\frac{i(i-1)}2}{k \choose n}_{q} {n \choose i}_{q} (q-1)^{i} x^i \partial_{M,q}^{k-n+i}(s)
\]
(for $n \leq k$ and $0$ otherwise).
\end{lem}

\begin{proof}
We proceed by induction on $k$ (the case $k=0$ being trivial).
Note that the case $k=1$ is essentially lemma \ref{difop}.
By definition,
\begin{align*}
\widetilde \partial_{M,q}^{k+1}(\xi^{[n]} \otimes s) &= (\widetilde \partial_{M,q} \circ \mathrm{Id} \otimes' \widetilde \partial_{M,q}^k \circ \Delta_{1,k} \otimes' \mathrm{Id}) (\xi^{[n]} \otimes s)
\\ &= (\widetilde \partial_{M,q} \circ \mathrm{Id} \otimes' \widetilde \partial_{M,q}^k) (1 \otimes \xi^{[n]} \otimes s + \xi \otimes \xi^{[n-1]} \otimes s)
\\ &= \widetilde \partial_{M,q}( 1 \otimes \widetilde \partial_{M,q}^k(\xi^{[n]} \otimes s)) + \widetilde \partial_{M,q}( \xi \otimes \widetilde \partial_{M,q}^k(\xi^{[n-1]} \otimes s))
\end{align*}
and we will compute separately both terms.
First of all, we have for $0 \leq n \leq k$,
\begin{align*}
\widetilde \partial_{M,q}( 1 \otimes \widetilde \partial_{M,q}^k(\xi^{[n]} \otimes s))&= \sum_{i = 0}^{n} q^{\frac{i(i-1)}2}{k \choose n}_{q} {n \choose i}_{q} (q-1)^{i} \partial_{M,q}(x^i \partial_{M,q}^{k-n+i}(s))
\\ &= \sum_{i = 1}^{n} q^{\frac{i(i-1)}2}{k \choose n}_{q}{n \choose i}_{q} (q-1)^{i} (i)_qx^{i-1}\partial_{M,q}^{k-n+i}(s)
\\ &+ \sum_{i = 0}^{n} q^{\frac{i(i-1)}2}{k \choose n}_{q}{n \choose i}_{q} (q-1)^{i} q^ix^i \partial_{M,q}^{k-n+i+1}(s)
\\ &= \sum_{i = 0}^{n-1} q^{\frac{i(i+1)}2}{k \choose n}_{q}{n \choose i+1}_{q} (q-1)^{i+1} (i+1)_qx^{i}\partial_{M,q}^{k-n+i+1}(s)
\\ &+ \sum_{i = 0}^{n} q^{\frac{i(i-1)}2}{k \choose n}_{q}{n \choose i}_{q} (q-1)^{i} q^ix^i \partial_{M,q}^{k-n+i+1}(s).
\end{align*}
We can rearrange the first sum:
\begin{align*}
q^{\frac{i(i+1)}2}{n \choose i+1}_{q} (q-1)^{i+1} (i+1)_q &= q^{\frac{i(i-1)}2}q^i{n \choose i}_{q} (n-i)_q(q-1)^{i+1}
\\&= q^{\frac{i(i-1)}2}q^i{n \choose i}_{q} (q^{n-i} -1)(q-1)^{i}
\\&= q^{\frac{i(i-1)}2} {n \choose i}_{q} (q^{n} -q^i)(q-1)^{i}.
\end{align*}
It follows that
\begin{align*}
\widetilde \partial_{M,q}( 1 \otimes \widetilde \partial_{M,q}^k(\xi^{[n]} \otimes s))= \sum_{i = 0}^{n} q^{\frac{i(i-1)}2} q^n {k \choose n}_{q}{n \choose i}_{q} (q-1)^{i} x^i \partial_{M,q}^{k-n+i+1}(s).
\end{align*}
We turn now to the second term and it follows from lemma \ref{difop} that, for $0 \leq n-1 \leq k$,
\begin{align*}
\widetilde \partial_{M,q}( \xi \otimes \widetilde \partial_{M,q}^k(\xi^{[n-1]} \otimes s)) &= \sum_{i = 0}^{n-1} q^{\frac{i(i-1)}2}{k \choose n-1}_{q}{n-1 \choose i}_{q} (q-1)^{i} x^i \partial_{M,q}^{k-n+i+1}(s)
\\&+ (q-1)x \sum_{i = 0}^{n-1} q^{\frac{i(i-1)}2}{k \choose n-1}_{q}{n-1 \choose i}_{q} (q-1)^{i} \partial_{M,q}(x^i \partial_{M,q}^{k-n+i+1}(s))
\\ &= \sum_{i = 0}^{n-1} q^{\frac{i(i-1)}2}{k \choose n-1}_{q}{n-1 \choose i}_{q} (q-1)^{i} x^i \partial_{M,q}^{k-n+i+1}(s)
\\&+ \sum_{i = 0}^{n-1} q^{\frac{i(i-1)}2}{k \choose n-1}_{q} {n-1 \choose i}_{q} (q-1)^{i+1} (i)_q x^{i} \partial_{M,q}^{k-n+i+1}(s)
\\&+ \sum_{i = 0}^{n-1} q^{\frac{i(i-1)}2}{k \choose n-1}_{q}{n-1 \choose i}_{q} (q-1)^{i+1} q^{i}x^{i+1} \partial_{M,q}^{k-n+i+2}(s)
\\ &= \sum_{i = 0}^{n-1} q^{\frac{i(i-1)}2}{k \choose n-1}_{q}{n-1 \choose i}_{q} (q-1)^{i} x^i \partial_{M,q}^{k-n+i+1}(s)
\\&+ \sum_{i = 0}^{n-1} q^{\frac{i(i-1)}2}{k \choose n-1}_{q} {n-1 \choose i}_{q} (q-1)^{i} (q^i-1) x^{i} \partial_{M,q}^{k-n+i+1}(s)
\\&+ \sum_{i = 1}^{n} q^{\frac{(i-2)(i-1)}2}{k \choose n-1}_{q}{n-1 \choose i-1}_{q} (q-1)^{i} q^{i-1}x^{i} \partial_{M,q}^{k-n+i+1}(s).
\end{align*}
Since
\[
q^{\frac{(i-2)(i-1)}2} q^{i-1} = q^{\frac{i(i-1)}2}
\quad \mathrm{and} \quad 
{n-1 \choose i}_{q} q^i + {n-1 \choose i-1}_{q} = {n \choose i}_{q},
\]
we finally get
\[
\widetilde \partial_{M,q}( \xi \otimes \widetilde \partial_{M,q}^k(\xi^{[n-1]} \otimes s)) = \sum_{i=0}^{n} q^{\frac{(i-1)i}2}{k \choose n-1}_{q}{n \choose i}_{q} (q-1)^{i} x^{i} \partial_{M,q}^{k-n+i+1}(s).
\]
Now, we use
\[
q^n {k \choose n}_{q} + {k \choose n-1}_{q} = {k+1 \choose n}_{q}
\]
in order to obtain
\[
\widetilde \partial_{M,q}^{k+1}(\xi^{[n]} \otimes s) = \sum_{i = 0}^{n} q^{\frac{i(i-1)}2}{k+1 \choose n}_{q} {n \choose i}_{q} (q-1)^{i} x^i \partial_{M,q}^{k-n+i+1}(s). \qedhere
\]
\end{proof}

\begin{lem} \label{lift}
If $M$ is a $ D_{A/R,q}$-module, then the structural map $D_{A/R,q} \to \mathrm{End}_{R}(M)$ lifts uniquely to an $A\langle \xi \rangle_{q}$-linear morphism of $R$-algebras
\[
\rho : \mathrm{Diff}_{q}(A) \simeq D_{A/R,q} \to \mathrm{Diff}_{q}(M).
\]
\end{lem}

\begin{proof}
The image of the generator $\partial_{q} \in D_{A/R,q}$ in $\mathrm{End}_{R}(M)$ is a $q$-derivation $\partial_{M,q}$ of $M$.
Moreover, we may observe that
\[
\forall f \in A, \quad \widetilde \partial_{M,q} \widetilde \circ f = \sigma(f)\ \widetilde \circ\ \widetilde \partial_{M,q} + \partial_{A,q}(f).
\]
To verify this formula, it is sufficient to check it on $\xi \otimes s$ for $s \in M$, but we have $\partial_{q}(f)(\xi \otimes s) = 0$ and
\begin{align*}
(\widetilde \partial_{M,q}\ \widetilde \circ\ f)(\xi \otimes s) & = \widetilde \partial_{M,q} (\theta_{1}(f)\xi \otimes s) \\& = \widetilde \partial_{M,q} (\sigma(f)\xi \otimes s) \\ &= \sigma(f)\widetilde \partial_{M,q} (\xi \otimes s).
\end{align*}
It therefore follows from lemma \ref{difop} that the structural map $D_{A/R,q} \to \mathrm{End}_{R}(M)$ lifts uniquely to an $A$-linear morphism of $R$-algebras $\rho : D_{A/R,q} \to \mathrm{Diff}_{q}(M)$ sending $\partial_{q}$ to $\widetilde \partial_{M,q}$.

It remains to show that $\rho$ is $A\langle \xi \rangle_{q}$-linear.
It follows from lemmas \ref{actd2} and \ref{actd} that, for $s \in M$, we have
\[
\forall k,n \in \mathbb N, \quad \rho(\partial_{q}^{k})(\xi^{[n]} \otimes s) = \rho\left(\xi^{[n]} \cdot \widetilde \partial_{q}^{k}\right)(1 \otimes s).
\]
By linearity, we see that 
\[
\forall P \in D_{A/R,q}, \forall \varphi \in A\langle \xi \rangle_{q}, \quad \rho(\widetilde P)(\varphi \otimes s) = \rho\left(\varphi \cdot \widetilde P\right)(1 \otimes s).
\]
Now, if $P \in D_{A/R,q}$ and $\varphi, \psi \in A\langle \xi \rangle_{q}$, we will have
\begin{align*}
(\varphi \cdot \rho(\widetilde P))(\psi \otimes s) &= \rho(\widetilde P)(\varphi\psi \otimes s)
\\&= \rho\left(\varphi\psi \cdot \widetilde P\right)(1 \otimes s)
\\&= \rho\left(\psi \cdot (\varphi \cdot \widetilde P)\right)(1 \otimes s)
\\&= \rho\left(\varphi \cdot \widetilde P\right)(\psi \otimes s). \qedhere
\end{align*}
\end{proof}

As a consequence, we see that $\rho$ induces for all $n \in \mathbb N$ an $A\langle \xi \rangle_{q}/I^{[n+1]}$-linear map
\[
\rho_n : \mathrm{Diff}_{q,n}(A) \to \mathrm{Diff}_{q,n}(M).
\]

The next notion is that of a $q$-Taylor structure whose infinite level analog is studied in section 5 of \cite{LeStumQuiros18}:

\begin{dfn}
A \emph{$q$-Taylor structure (of level zero)} on an $A$-module $M$ is a compatible family of semilinear maps ($A$-linear if we consider the target as an $A$-module via the $q$-Taylor map) 
\[
\theta_{n} : M \to M \otimes_{A} A\langle \xi \rangle_{q}/I^{[n+1]},
\]
such that
\begin{equation} \label{cond1}
\forall s \in M, \theta_{0}(s) = s \otimes 1
\end{equation}
and
\begin{equation} \label{cond2}
\forall m, n \in \mathbb Z_{\geq 0}, \quad (\theta_{n} \otimes \mathrm{Id}_{A\langle \xi \rangle_{q}/I^{[m+1]}}) \circ \theta_{m} = (\mathrm{Id}_{M} \otimes \Delta_{n,m}) \circ \theta_{m+n}.
\end{equation}
\end{dfn}

We can make explicit the semilinear condition, which sounds very natural:
\[
\forall f \in A, \forall s \in M, \quad \theta_{n}(fs) = \theta_{n}(f) \theta_{n}(s).
\]

\begin{prop} \label{taylD}
A $q$-Taylor structure on an $A$-module $M$ is equivalent to a $ D_{A/R,q}$-module structure through the formula
\begin{equation} \label{Tayldiv}
\forall s \in M, \quad \theta_{n}(s) = \sum_{k=0}^n \partial^k_{M,q}(s) \otimes \xi^{[k]}.
\end{equation}
\end{prop}

\begin{proof}
In general, if $X$ and $Y$ are two $A$-modules, then we have an adjunction
\[
X \to \mathrm{Hom}_A(Y,M) \quad  \Leftrightarrow \quad X \otimes_A Y \to M \quad \Leftrightarrow \quad Y \to \mathrm{Hom}_A(X,M)
\]
given by
\[
x \mapsto (y \mapsto \langle x, y \rangle) \quad \Leftrightarrow \quad x\otimes y \mapsto \langle x, y \rangle \quad \Leftrightarrow \quad y \mapsto (x \mapsto \langle x, y \rangle).
\]
If $B$ is an $A$-module and $B'$ is an $A$-bimodule, then we can apply this to $X = B' \otimes_A M$ (in which we use the right structure for the tensor product and the left structure in order to turn the tensor product into an $A$-module) and $Y = \mathrm{Hom}_A(B,A)$ and we get
\begin{equation} \label{adju}
B' \otimes_A M \to \mathrm{Hom}_A(\mathrm{Hom}_{A}(B,A),M) \quad \Leftrightarrow \quad  \mathrm{Hom}_A(B,A) \to \mathrm{Hom}_A(B' \otimes_{A} M,M)
\end{equation}
given by
\[
\varphi \otimes s \mapsto (P \mapsto \langle \varphi \otimes s, P \rangle) \quad \Leftrightarrow \quad  P \mapsto (\varphi \otimes s \mapsto \langle \varphi \otimes s, P \rangle).
\]
When $B$ is free over $A$, we also have an isomorphism
\[
M \otimes_A B \simeq \mathrm{Hom}_{A}(\mathrm{Hom}_A(B, A),M), \quad s \otimes \varphi \mapsto (P \mapsto P(\varphi)s)
\]
from which we deduce an adjunction
\[
B' \otimes_A M \overset \epsilon \to M \otimes_A B  \quad \Leftrightarrow \quad  \mathrm{Hom}_A(B,A) \overset \rho \to \mathrm{Hom}_A(B' \otimes_{A} M,M).
\]
One can make this explicit: if $\epsilon(\varphi \otimes s) = \sum s_k \otimes \varphi_k$, then adjunction \eqref{adju} reads
\[
\varphi \otimes s \mapsto \left(P \mapsto \sum P(\varphi_k)s_k) \right)
 \quad \Leftrightarrow \quad 
P \mapsto (\varphi \mapsto \rho(P)(\varphi \otimes s)).
\]
In other words
\[
\epsilon(\varphi \otimes s) = \sum s_k \otimes \varphi_k  \quad \Leftrightarrow \quad  \rho(P)(\varphi \otimes s) = \sum P(\varphi_k)s_k.
\]
We now apply this to the case where $B$ (resp. $B'$) is equal to $A\langle \xi \rangle_{q}/I^{[n+1]}$ with its canonical $A$-module structure (resp. the $A$-bimodule structure given by the canonical structure on the left and the Taylor structure on the right).
In this situation, we have
\[
B' \otimes_{A} M = A\langle \xi \rangle_{q}/I^{[n+1]} \otimes'_{A} M
\]
and we write $\epsilon_n$ and $\rho_n$ instead of $\epsilon$ and $\rho$.
One easily sees that $\epsilon_n$ is $A\langle \xi \rangle_{q}/I^{[n+1]}$-linear if and only if $\rho_n$ is $A\langle \xi \rangle_{q}/I^{[n+1]}$-linear.
Therefore, an $A$-linear map
\[
\theta_{n} : M \to M \otimes_{A} A\langle \xi \rangle_{q}/I^{[n+1]}
\]
extends uniquely to an $A\langle \xi \rangle_{q}/I^{[n+1]}$-linear map
\[
\epsilon_{n} : A\langle \xi \rangle_{q}/I^{[n+1]} \otimes'_{A} M \to M \otimes_{A} A\langle \xi \rangle_{q}/I^{[n+1]},
\]
which in turn corresponds by adjunction to a unique $A\langle \xi \rangle_{q}/I^{[n+1]}$-linear map
\[
\xymatrix{
\mathrm{Diff}_{q,n}(A) \ar[r]^{\rho_n} \ar@{=}[d] & \mathrm{Diff}_{q,n}(M) \ar@{=}[d] \\ \mathrm{Hom}_{A}\left(A \langle \xi \rangle_{q}/I^{[n+1]} , A\right) \ar[r] & \mathrm{Hom}_{A}\left(A \langle \xi \rangle_{q}/I^{[n+1]} \otimes'_{A} M, M \right).}
\]
This shows that the compatible families $\{\theta_n\}_{n \in \mathbb N}$ of $A$-linear maps and $\{\rho_n\}_{n\in \mathbb N}$ of $A \langle \xi \rangle_{q}/I^{[n+1]} $-linear maps determine each other uniquely.
Moreover, since
\[
\epsilon_n(1 \otimes s) = \sum s_k \otimes \xi^{[k]}  \quad \Leftrightarrow \quad  \forall l \in \mathbb N, \rho_n(\widetilde \partial_q^l)(1\otimes s) = \sum \widetilde \partial_q^l(\xi^{[k]})s_k,
\]
we will have $s_k = \rho_n(\widetilde \partial_q^k)(1\otimes s)$.
In other words,
\begin{equation} \label{comprho}
\forall s \in M, \quad \theta_{n}(s) = \sum_{k=0}^n \rho_n( \widetilde{\partial_{q}^{k}} )(1 \otimes s) \otimes \xi^{[k]}
\end{equation}
and formula \eqref{Tayldiv} will therefore hold in the end.
Taking the limit on the family $\{\rho_n\}_{n \in \mathbb N}$, we obtain an $R$-linear morphism
\[
\rho : D_{A/R,q} \simeq \mathrm{Diff}_{q}(M) \to \mathrm{Diff}_{q}(M).
\]
Using formula \eqref{comprho}, condition \eqref{cond1} reads $\rho_0(1) = \mathrm{Id}_M$ and condition \eqref{cond2} reads
\[
 \sum_{k,l=0}^{n,m} \rho_{n}(\widetilde{\partial_{q}^{k}})(\rho_m(\widetilde{\partial_{q}^{l}})(1 \otimes s)) \otimes \xi^{[k]} \otimes \xi^{[l]} = \sum_{k,l=0}^{n,m} \rho_{n+m}(\widetilde{\partial_{q}^{k+l}}) (1 \otimes s) \otimes \xi^{[k]} \otimes \xi^{[l]}.
\]
We know from lemma \ref{goodbasis} that $\left\{\xi^{[k]} \otimes \xi^{[l]} \right \}_{k,l=0}^{n,m}$ is a basis for $A\langle\xi\rangle_{q}/I^{[n+1]} \otimes'_A A\langle\xi\rangle_{q}/I^{[m+1}$ and conditions \eqref{cond1} and \eqref{cond2} are therefore equivalent to $\rho$ being a morphism of rings.
We can therefore conclude with the help of lemma \ref{lift}.
\end{proof}

One can also define a \emph{$q$-stratification (of level zero)} as a compatible family of $A\langle \xi \rangle_{q}/I^{[n+1]}$-linear \emph{isomorphisms}
\[
\epsilon_{n} : A\langle \xi \rangle_{q}/I^{[n+1]} \otimes'_{A} M \simeq M \otimes_{A} A\langle \xi \rangle_{q}/I^{[n+1]}
\]
satisfying a normalization and a cocycle condition of the same type as \eqref{cocyc}.
Clearly, a $q$-stratification induces a $q$-Taylor structure.
Moreover we see as in the proof of proposition \ref{taylD} that
\[
\epsilon_n(\varphi \otimes s) = \sum s_k \otimes \xi^{[k]}  \quad \Leftrightarrow \quad  \forall l \leq n,\widetilde \partial_{M,q}^l(\varphi \otimes s) = \sum \widetilde \partial_q^l(\xi^{[k]})s_k = s_l
\]
and it follows that
\[
\forall \varphi \in A\langle \xi \rangle_{q}/I^{[n+1]}, \forall s \in M, \quad \epsilon_n(\varphi \otimes s) = \sum_{k=0}^n \widetilde \partial_{M,q}^k(\varphi \otimes s) \otimes \xi^{[k]}.
\]
Said differently, we obtain the following improvement on \eqref{case1}:
\[
\widetilde {\partial}_{M,q}^n : A\langle\xi\rangle_{q}/I^{[n+1]} \otimes_A' M \overset{\epsilon_n} \longrightarrow M \otimes_A A\langle\xi\rangle_{q}/I^{[n+1]} \overset{\mathrm{Id} \otimes \widetilde \partial_{A,q}^n} \longrightarrow M.
\]
However, there is no flip map on $A\langle \xi \rangle_{q}/I^{[n+1]}$ (or on $A[\xi]/\xi^{(n+1)}$ either) because $\tau (I^{[n+1]}) \not\subset I^{[n+1]}$, and it follows that a $q$-Taylor structure will not extend to a $q$-stratification in general.
Actually, if we are given a $q$-stratification on an $A$-module $M$, then one can verify that
\[
\epsilon_1^{-1}(s \otimes \xi) \equiv \sum_{k=0}^N (-1)^k(q-1)^kx^k\xi \otimes \partial_q^k(s) \mod (q-1)^{N+1}.
\]
Clearly, if $M$ is a general $D_{A/R,q}$-module and $(q-1)$ is not nilpotent, then this sum could be infinite.
It is therefore necessary to consider the following classical topological condition:

\begin{dfn}
A $D_{A/R,q}$-module $M$ is \emph{topologically quasi-nilpotent} if
\[
\forall s \in M, \quad \lim_{k \to+\infty} \partial_{M,q}^k(s)=0.
\]
\end{dfn}

We may also say that the corresponding connection, $q$-Taylor structure, etc. is \emph{topologically quasi-nilpotent}.

Recall that we introduced the notion of hyper-$q$-stratification in definition \ref{hyperqpd}.
Then, we have:

\begin{prop} \label{bigequi}
If $M$ is a finitely presented flat $A$-module, then it is equivalent to give $M$ the structure of a topologically quasi-nilpotent $D_{A/R,q}$-module or to endow it with a hyper-$q$-stratification.
\end{prop}

\begin{proof}
Thanks to propositions \ref{taylD} and \ref{eqTay}, it is sufficient to show that a $q$-Taylor structure is equivalent to a hyper-$q$-Taylor structure.
By reduction modulo $I^{[n+1]}$ for all $n \in \mathbb N$, the latter will automatically induce the former.
Conversely, since $M$ is finitely presented over $A$, a $q$-Taylor structure on $M$ provides a \emph{formal $q$-Taylor map}
\[
\hat \theta : M \to \varprojlim(M \otimes_{A} A\langle \xi \rangle_{q}/I^{[n+1]} )\simeq M \otimes_{A} A\langle \langle \xi \rangle \rangle_{q}, \quad s \mapsto \sum_{k=0}^{+\infty} \partial_{M,q}^{k}(s) \otimes \xi^{[k]},
\]
where the last isomorphism results from \cite[\href{https://stacks.math.columbia.edu/tag/059K}{Tag 059K}]{stacks-project}
since we have an isomorphism of free $A$-modules $A\langle \xi \rangle_{q}/I^{[n+1]} \simeq \prod_{k=0}^n A\xi^{[k]}$.
Since we also assume that $M$ is flat, it is a direct summand of a finite free module (\cite[\href{https://stacks.math.columbia.edu/tag/00NX}{Tag 00NX}]{stacks-project}) and we have
\[
\sum_{k=0}^{+\infty} \partial_{M,q}^{k}(s) \otimes \xi^{[k]} \in M \otimes_{A} \widehat{A\langle \xi \rangle}_{q} \Leftrightarrow \lim_{k\to +\infty} \partial_{M,q}^{k}(s) = 0.
\]
Thus, we see that the formal $q$-Taylor map factors through $M \otimes_{A} \widehat{A\langle \xi \rangle}_{q}$ if and only if the $q$-Taylor structure is topologically quasi-nilpotent and it remains to check that the induced map $\theta : M \to M \otimes_{A} \widehat{ A\langle \xi \rangle}_{q}$ is a hyper-$q$-Taylor structure.
Since $(\mathrm{Id}_{M} \otimes e) \circ \theta = \mathrm{Id}_{M}$, we only have to check that
\[
(\theta \otimes \mathrm{Id}_{\widehat {A\langle\xi\rangle_{q}}}) \circ \theta = (\mathrm{Id}_{M} \otimes \Delta) \circ \theta : M \to M \otimes_{A} \widehat{A\langle \xi \rangle}_{q} \widehat \otimes_{A} \widehat{A\langle \xi \rangle}_{q}
\]
and it is sufficient to make sure that the canonical map
\[
M \otimes_{A} \widehat{A\langle\xi\rangle_{q}} \widehat\otimes'_{A} \widehat{A\langle\xi\rangle_{q}} \to M \otimes_{A} A \langle\langle \xi \rangle\rangle_{q} \widehat\otimes'_{A} A \langle\langle \xi \rangle\rangle_{q}
\]
is injective.
Since $M$ is flat, this follows from the fact that, in the following diagram, the bottom map is injective and the flip maps are bijective.
\[
\xymatrix{
\widehat{A\langle\xi\rangle_{q}} \widehat\otimes'_{A} \widehat{A\langle\xi\rangle_{q}} \ar[r] \ar[d]^{ \tau \otimes \mathrm{Id}}& A \langle\langle \xi \rangle\rangle_{q} \widehat\otimes'_{A} A \langle\langle \xi \rangle\rangle_{q} \ar[d]^{ \tau \otimes \mathrm{Id}}
\\
\widehat{A\langle\xi\rangle_{q}} \widehat\otimes_{A} \widehat{A\langle\xi\rangle_{q}} \ar[r] & A \langle\langle \xi \rangle\rangle_{q} \widehat\otimes_{A} A \langle\langle \xi \rangle\rangle_{q}
\\
\widehat \bigoplus_{i,j \in \mathbb N} A(\xi^{[i]} \otimes \xi^{[j]}) \ar@{^{(}->}[r] \ar[u]^\simeq & \prod_{i,j \in \mathbb N} A (\xi^{[i]} \otimes \xi^{[j]}). \ar[u]^\simeq
\qedhere\qed} 
\]
\end{proof}

\section{$q$-crystals}

All $\delta$-rings are assumed to live over the local ring $\mathbb Z[q]_{(p,q-1)}$ (with $q$ of rank one) and they are endowed with their $(p, q-1)$-adic topology.
We fix a morphism of bounded $q$-PD-pairs $(R, \mathfrak r) \to (A,\mathfrak a)$ with $A$ complete with fixed topologically \'etale coordinate $x$.
We also assume that $\mathfrak a$ is closed in $A$ so that $\overline A$ is also complete. 

The absolute (big) \emph{$q$-crystalline site} $q\mathrm{-CRIS}$ is the category opposite\footnote{We will always write morphisms of pairs in the usual way and call it an opposite morphism if we consider it as a morphism in the $q$-crystalline site.} to the category of complete \emph{bounded} \footnote{As in \cite{BhattScholze22}, we could 
add some other technical assumptions: see their definition 16.2.} $q$-PD-pairs $(B, J)$ (or equivalently complete $q$-PD-thickenings $B \twoheadrightarrow \overline B$ with $B$ bounded).
We may actually consider the slice category $q\mathrm{-CRIS}_{/R}$ over $(\widehat R, \widehat {\mathfrak r})$: an object is a morphism $(R, \mathfrak r) \to (B,J)$ to a complete bounded $q$-PD-pair (and a morphism is the opposite of a morphism of $q$-PD-pairs which is compatible with the structural maps).
Now, if we denote by $\mathrm{FSch}$ the category of $(p, q-1)$-adic formal schemes (over $\mathrm{Spf}\left(\mathbb Z_{p}[[q-1]]\right)$), then we may consider the functor 
\[
q\mathrm{-CRIS} \to \mathrm{FSch}, \quad (B, J) \mapsto \mathrm{Spf}(\overline B) \quad (\mathrm{with}\ \overline B = B/J).
\]
If $\mathcal X$ is a $(p,q-1)$-adic formal scheme, then the absolute $q$-crystalline site $q\mathrm{-CRIS}(\mathcal X)$ of $\mathcal X$ will be the corresponding fibered site over $\mathcal X$: an object is a pair made of a complete bounded $q$-PD-pair $ (B,J)$ and a morphism $\mathrm{Spf}(\overline B) \to \mathcal X$ of formal schemes.

We will actually mix both constructions and consider, if $\mathcal X$ is a $(p,q-1)$-adic formal $\overline R$-scheme, the $q$-crystalline site $q\mathrm{-CRIS}(\mathcal X/R)$ of $\mathcal X/R$: an object is a pair made of a morphism $(R, \mathfrak r) \to (B,J)$ to a complete bounded $q$-PD-pair together with a morphism $\mathrm{Spf}(\overline B) \to \mathcal X$ of formal $\overline R$-schemes.
We will write
\[
q\mathrm{-CRIS}(\overline A/R):= q\mathrm{-CRIS}( \mathrm{Spf}(\overline A)/R).
\]
For the moment, we endow the category $q\mathrm{-CRIS}(\mathcal X/R)$ with the \emph{coarse} topology, so that a sheaf is simply a presheaf, and we denote by $(\mathcal X/R)_{q\mathrm{-CRIS}}$ the corresponding topos\footnote{Unlike in \cite{BhattScholze22}, we follow the standard notations for sites and topos.}.
A sheaf (of sets) $E$ on $q\mathrm{-CRIS}(\mathcal X/R)$ is thus a family of sets $E_{B}$ together with a compatible family of maps $E_{B} \to E_{B'}$ for any morphism opposite to $(B,J) \to (B', J')$ in $q\mathrm{-CRIS}(\mathcal X/R)$.
In particular, we may consider the sheaf of rings $\mathcal O_{\mathcal X/R,q}$ that sends $B$ to itself (we will write $\mathcal O_{\overline A/R,q}$ when $\mathcal X = \mathrm{Spf}(\overline A)$ as above).
A sheaf of $\mathcal O_{\mathcal X/R,q}$-modules is a family of $B$-modules $E_{B}$ endowed with a compatible family of semilinear maps $E_{B} \to E_{B'}$, or better, of linear maps $B' \otimes_{B} E_{B} \to E_{B'}$, called the \emph{transition} maps.

\begin{dfn}
If $\mathcal X$ is a $(p,q-1)$-adic formal $\overline R$-scheme, then a \emph{$q$-crystal} on $\mathcal X/R$ is a sheaf of $\mathcal O_{\mathcal X/R,q}$-modules whose transition maps are all bijective.
\end{dfn}

\begin{rmks}
\begin{enumerate}
\item
An $\mathcal O_{\mathcal X/R,q}$-module $E$ is finitely presented if and only if it is a \emph{$q$-crystal} and all $E_{B}$'s are finitely presented $B$-modules.
\item
An $\mathcal O_{\mathcal X/R,q}$-module $E$ is flat if and only all $E_{B}$'s are flat $B$-modules.
\end{enumerate}
\end{rmks}

Since $x$ is a topologically \'etale coordinate on $A$, there exists a unique $\delta$-structure on $A$ such that $x$ has rank one (use lemma 2.18 of \cite{BhattScholze22} for example).
We endow the polynomial ring $A[\xi]$ with the \emph{symmetric} $\delta$-structure so that $x + \xi$ also has rank one.
The ring $P := A \otimes_{R} A$ comes with its tensor product $\delta$-structure and the canonical map $A[\xi] \to P$ is a morphism of $\delta$-rings.
It follows that $P$ has a Frobenius $\phi$, but, as in section \ref{deltwist}, if we set
\[
A' := R {}_{{}_{\phi}\nwarrow}\!\!\widehat\otimes_{R} A \quad \mathrm{and} \quad P' := A' \otimes_{R} A' = R {}_{{}_{\phi}\nwarrow}\!\!\otimes_{R} P,
\]
we could as well consider the relative Frobenius $F : P' \to P$ (which is a morphism of $\delta$-algebras) as we did in section 7 of \cite{GrosLeStumQuiros19}.

\begin{lem} \label{coinc}
If two morphisms of $\delta$-pairs $u_{1}, u_{2}: (P, I) \to (B,J)$ to a complete $q$-PD-pair coincide when restricted to $A[\xi]$, then they must be equal.
\end{lem}

\begin{proof}
The statement means that any morphism of $\delta$-pairs
\[
(P, I) \otimes_{(A[\xi], \xi)} (P,I) \to (B,J)
\]
will factor (necessarily uniquely) through the multiplication map
\[
(P, I) \otimes_{(A[\xi], \xi)} (P,I) \to (P, I).
\]
Since $x$ is a topologically \'etale coordinate on $A$, if we denote by $N$ the kernel of multiplication from $Q := A \otimes_{R[x]} A$ to $A$, then there exists a split exact sequence
\[
0 \to \widehat N \to \widehat Q \to A \to 0
\]
of \emph{rings} (see \cite[\href{https://stacks.math.columbia.edu/tag/02FL}{Tag 02FL}]{stacks-project} in the algebraic setting), and therefore also of $\delta$-rings thanks to lemma \ref{factdel}.
Therefore, letting $M := A \otimes_{R} N$, there exists a sequence of isomorphisms
\[
P \widehat \otimes_{A[\xi]} P \simeq A \widehat \otimes_{R} Q \simeq A \widehat \otimes_{R} (N \times A) \simeq \widehat M \times \widehat P.
\]
Actually, there even exists an isomorphism of exact sequences
\[
\xymatrix{
0 \ar[r] & P \widehat \otimes_{A[\xi]} I + I \widehat \otimes_{A[\xi]} P \ar[r] \ar[d]^{\simeq} & P \widehat \otimes_{A[\xi]} P \ar[r] \ar[d]^{\simeq} & A \ar[r]\ar[d]^{\simeq} & 0 \\
0 \ar[r] & \widehat M \times \widehat I \ar[r] & \widehat M \times \widehat P \ar[r] & A \ar[r]& 0,
}
\]
where the upper right map is total multiplication and the lower right map is composition of projection and multiplication.
In particular, there exists an isomorphism of $\delta$-pairs
\[
(P, I) \widehat \otimes_{(A[\xi], \xi)} (P,I) = (P \widehat \otimes_{A[\xi]} P, P \widehat \otimes I + I \widehat \otimes P) \simeq ( \widehat M \times \widehat P, \widehat M \times \widehat I) = ( \widehat M, \widehat M) \times ( \widehat P, \widehat I).
\]
Let us also note that there exists a commutative diagram
\[
\xymatrix{
( P, I) \widehat \otimes_{(A[\xi], \xi)} (P,I) \ar[r] \ar[d]^{\simeq} &(\widehat P, \widehat I) \ar@{=}[d]\\
( \widehat M,\widehat M) \times ( \widehat P, \widehat I) \ar[r] &(\widehat P,\widehat I),
}
\]
where the upper map is multiplication and the lower map is the projection.
To finish the proof, it is therefore sufficient to show that any morphism of $\delta$-pairs
\[
u : (\widehat M, \widehat M) \times (\widehat P, \widehat I) \to (B,J)
\]
factors through $( \widehat P, \widehat I)$.
But, for any such $u$, there exists a decomposition $u = u' \times u''$ with $u' : ( \widehat M, \widehat M) \to (B',J')$ and $u'' : ( \widehat P, \widehat I) \to (B'',J'')$ as a product of morphisms of $\delta$-rings (use lemma \ref{factdel} again) and $u'$ factors necessarily through the \emph{complete} $q$-PD-envelope of $( \widehat M, \widehat M)$, which is the zero ring (because we use completions).
\end{proof}

The morphism
\[
\widetilde \theta: (P, I) \to \left(\widehat {A\langle \xi \rangle_{q}}, \widehat I^{[1]}\right)
\]
obtained by linear extension from the $q$-Taylor map of definition \ref{qTay} is clearly a morphism of $\delta$-pairs and we can show that it is universal:

\begin{thm} \label{univpro}
The $q$-PD-pair $\left(\widehat {A\langle\xi\rangle_{q}}, \widehat I^{[1]}\right)$ is the complete $q$-PD-envelope of $(P, I)$.
\end{thm}

\begin{proof}
We give ourselves a map $u : (P, I) \to (B,J)$ to a complete $q$-PD-pair.
Thanks to theorem \ref{algun} and proposition \ref{bound}, the restriction of $u$ to $A[\xi]$ extends uniquely to a morphism $v : \left(\widehat {A\langle\xi\rangle_{q}}, \widehat I^{[1]}\right) \to (B,J)$.
Moreover, the diagram
\[
\xymatrix{
(A[\xi], \xi) \ar[r] \ar[d] & \left(\widehat {A\langle\xi\rangle_{q}}, \widehat I^{[1]}\right) \ar@{=}[d]\\
(P, I) \ar[r]^-{\widetilde \theta} & \left(\widehat {A\langle\xi\rangle_{q}}, \widehat I^{[1]}\right) 
}
\]
is commutative.
We may then apply lemma \ref{coinc} to the maps $u_{1} = u$ and
\[
\xymatrix{
u_{2} : \left( P, I\right) \ar[r]^-{{\widetilde \theta}} & \left(\widehat {A\langle\xi\rangle_{q}}, \widehat I^{[1]}\right) \ar[r]^-{v} & (B,J).
}\qedhere
\] 
\end{proof}

We will need a slight generalization of the theorem \ref{univpro} and we consider now
\[
P(r) := \underbrace{A \otimes_R \cdots \otimes_R A}_{r+1\ \mathrm{times}} \simeq \underbrace{P \otimes'_A \cdots \otimes'_A P}_{r\ \mathrm{times}}
\]
and denote by $I(r)$ the kernel of multiplication $P(r) \to A$.
We also define
\begin{align*} 
A\langle \xi_1, \ldots, \xi_r \rangle_{q}^\wedge &:= \widehat{A\langle \xi_1}\rangle_{q} \widehat \otimes_A' \cdots \widehat \otimes_A' \widehat{A\langle \xi_{r} \rangle}_{q}
\\ & \simeq \left(A\langle \xi_1, \ldots, \xi_{r-1} \rangle^\wedge_{q} \langle \xi_{r} \rangle_{q}\right)^\wedge.
\end{align*}

\begin{cor} \label{sevvar}
$A\langle \xi_1, \ldots, \xi_r \rangle_{q}^\wedge$ is the complete $q$-PD-envelope of $I(r)$ in $P(r)$.
\end{cor}

\begin{proof}
This is a direct consequence of lemma \ref{tensor}.
The only thing to check is that $A\langle \xi_1, \ldots, \xi_r \rangle_{q}^\wedge$ is bounded and this is shown by induction:
since $A\langle \xi_{r} \rangle_{q}$ is flat over $A$, the map
\[
A\langle \xi_1, \ldots, \xi_{r-1} \rangle_{q}^\wedge \to A\langle \xi_1, \ldots, \xi_{r-1} \rangle_{q}^\wedge \otimes'_A A\langle \xi_r \rangle_{q}
\]
is flat and we may then apply successively lemma \ref{bounded} and corollary \ref{compbound}.
\end{proof}

\begin{rmks}
\begin{enumerate}
\item
If $\mathfrak a$ is a $q$-PD-ideal in $A$, then $\overline {\mathfrak a}^{\mathrm {cl}} \oplus \widehat I^{[1]}$ is a $q$-PD-ideal in $\widehat {A\langle\xi\rangle_{q}}$ and $\left(\widehat {A\langle\xi\rangle_{q}}, \overline {\mathfrak a}^{\mathrm {cl}} \oplus \widehat I^{[1]} \right)$ is the complete $q$-PD-envelope of $(P, \mathfrak a \oplus I)$.
This follows from lemma \ref{tensor} applied to $(A,0) \to (A, \mathfrak a)$ and $(A,0) \to (P, I)$, or can be proved directly.
\item 
For the same reason, il follows from corollary \ref{sevvar} that
\[
\left(\widehat {A\langle\xi\rangle_{q}} \widehat \otimes_{A}' \widehat {A\langle\xi\rangle_{q}}, \overline {\mathfrak a}^{\mathrm {cl}} \oplus \widehat {A\langle\xi\rangle_{q}} \widehat \otimes'_{A} I^{[1]} + I^{[1]} \widehat \otimes'_{A} \widehat {A\langle\xi\rangle_{q}} \right)
\]
is the complete $q$-PD-envelope of $(P \otimes'_{A} P, \mathfrak a \oplus P \otimes'_{A} I + I \otimes'_{A} P )$.
\end{enumerate}
\end{rmks}

\begin{thm} \label{mainres}
There exists a functor $E \mapsto E_{A}$ from finitely presented flat $q$-crystals on $\overline A/R$ to finitely presented flat topologically quasi-nilpotent $D_{A/R,q}$-modules.
\end{thm}

\begin{proof}
There exists an obvious functor from $q$-crystals on $\overline A/R$ to hyper-$q$-stratified modules on $A/R$ obtained by composing the transition maps
\[
\widehat {A\langle\xi\rangle_{q}} \otimes'_{A} E_{A} \simeq E_{\widehat {A\langle\xi\rangle_{q}}} \simeq E_{A} \otimes_{A} \widehat {A\langle\xi\rangle_{q}}.
\]
If $E$ is flat finitely presented, then $E_{A}$ may be seen by proposition \ref{bigequi} as a finitely presented flat topologically quasi-nilpotent $D_{A/R,q}$-module.
\end{proof}

\begin{rmks}
\begin{enumerate}
\item
In a forthcoming article, we will explain what happens when we endow the $q$-crystalline site with the flat topology.
We are confident that $(A, \mathfrak a)$ will then cover $q\mathrm{-CRIS}(\overline A/R)$ for this topology and, as a consequence, the functor in theorem \ref{mainres} will be then an equivalence.
Note that this result could also be deduced from the recent work of Chatzistamatiou in \cite{Chatzistamatiou20} and partly recovered from the preprint \cite{MorrowTsuji20} of Morrow and Tsuji.
\item
There should also exist a comparison theorem in cohomology. One approach to try to prove such a result would be to investigate more closely the relation between the \v Cech-Alexander complexes used in \cite{BhattScholze22}, 16.13, and our constructions.
A second possibility would be to develop a theory of linearization of $q$-differential operators.
\end{enumerate}
\end{rmks}

\addcontentsline{toc}{section}{References}
\printbibliography

\Addresses

\end{document}